%% file: jcomp-template.tex
\documentclass[times,final]{elsarticle}

\usepackage{jcomp}
\usepackage{framed,multirow}

\usepackage{amssymb}
\usepackage{sansmath}

\usepackage{url}
\usepackage{xcolor}
\definecolor{newcolor}{rgb}{.8,.349,.1}

\usepackage{import}
\subimport{praemble/}{hft-paper-praemble-main.tex}

\usepackage{float} 
\usepackage{comment}
\theoremstyle{definition}
\newtheorem{theorem}{Theorem}

\newtheorem{proposition}{Proposition}
\newtheorem{corollary}{Corollary}

\newcommand\LapLm[1][]{\mat{\Delta}_{#1}}

\def\Nm{\mat{N}}
\def\Dm{\mat{D}}
\def\Sm{\mat{S}}
\def\Dxm{\mat{D}^*}

\journal{Journal of Computational Physics}

\begin{document}

\verso{V. Giunzioni \textit{et al.}}

\begin{frontmatter}

\title{On a Calder\'{o}n preconditioner for the symmetric formulation of the electroencephalography forward problem without barycentric refinements}%

\author[1]{Viviana \snm{Giunzioni}}
  
\author[2]{John E. \snm{Ortiz G.}
}

\author[3]{Adrien \snm{Merlini}}

\author[4]{Simon B. \snm{Adrian}}

\author[1]{Francesco P. \snm{Andriulli}\corref{cor1}}
\cortext[cor1]{Corresponding author: 
  \ead{francesco.andriulli@polito.it}.}
  
\address[1]{Department of Electronics and Telecommunications, Politecnico di Torino, 10129 Turin, Italy}
\address[2]{Electronic Engineering Department, Universidad de Nariño, Pasto 520002, Colombia}
\address[3]{Microwaves Department, IMT Atlantique, 29238 Brest, France}
\address[4]{Fakultät für Informatik und Elektrotechnik, Universität Rostock, 18051 Rostock, Germany}

\received{1 May 2013}
\finalform{10 May 2013}
\accepted{13 May 2013}
\availableonline{15 May 2013}
\communicated{S. Sarkar}

\begin{abstract}

We present a Calder\'{o}n preconditioning scheme for the symmetric formulation of the forward electroencephalographic (EEG) problem that cures both the dense-discretization and the high-contrast breakdown.
Unlike existing Calderón schemes presented for the EEG problem, it is refinement-free, that is,
the electrostatic integral operators are not discretized with basis functions defined on the barycentrically-refined dual mesh. In fact, in the preconditioner, we reuse the original system matrix thus reducing computational burden.
Moreover, the proposed formulation gives rise to a symmetric, positive-definite system of linear equations, which allows the application of the conjugate gradient method, an iterative method that exhibits a smaller computational cost compared to other Krylov subspace methods applicable to non-symmetric problems. 
Numerical results corroborate the theoretical analysis and attest of the efficacy of the proposed preconditioning technique on both canonical and realistic scenarios.
\end{abstract}


\end{frontmatter}


\section{Introduction} \label{sec:introduction}

Brain imaging techniques aim at fully determining the inner neural activity in terms of location, orientation, and intensity of the primary current, starting from some direct or indirect measurements of its correlated effects \citep{raichle2006brain,baillet2001electromagnetic}. Among them, source localization algorithms based on electroencephalographic (EEG) data are widely appreciated because of their high temporal resolution \citep{michel2019eeg,michel2012utilization} and of their compatibility with other imaging strategies, such as magnetoencephalography (MEG) \citep{schmitt2001numerical}, magnetic resonance imaging (MRI) \citep{benar2007singletrial, jorge2015highquality}, and positron emission tomography (PET) \citep{shah2013advances}. This technology aims at reconstructing the equivalent volume brain sources from the measurement of the resulting potential distribution at the scalp \citep{koessler2010source, plummer2008eeg}, which is known as the inverse EEG problem. Another inverse problem relying on EEG modelling is the one of inferring the electrical parameters of the biological tissues, that is, the electrical conductivity and the permittivity, from surfacic electroencephalographic measurements, known as electrical impedance tomography (EIT) \citep{cheney1999electrical}. 

These problems are usually addressed through multiple iterative solutions of the forward EEG problem \citep{grech2008review, darbas2019review, fuchs2001boundary}, which is the evaluation of the voltage function at the scalp resulting from a known electric activity inside the head \citep{hallez2007review, he1987electric}. It follows, on the one hand, that the accuracy achievable in the resolution of the inverse problem is strictly limited by the one of the forward EEG problem; on the other hand, the resolution time is also affected by the complexity of the numerical model describing the forward problem \citep{akalinacar2013effects}. These considerations evidence the crucial importance of determining realistic source and head models and defining affordable and efficient numerical schemes for the approximation of the electrostatic potential.

As far as the source modelling is concerned, equivalent currents, formed by the superposition of the effects of a number of infinitesimal dipoles, are widely considered to be an adequate representation of the electric brain activity whose effects are observed from the scalp \citep{demunck1988mathematical}, so that the majority of available numerical schemes relies on this approximation \citep{darbas2019review}. Many disparate head models have instead been developed and are still under active investigation \citep{darbas2019review}. They can be subdivided in two main categories. The methods of one of the categories assume that the electrical properties of the biological tissues constituting the head vary continuously so that a discretization of the full head volume is needed to numerically solve the problem; the formulations adopted in this case are both differential, traceable to the class of finite element method \citep{awada1997computational} or finite difference methods \citep{bruno2006fdm}, and integral \citep{rahmouni2017two,henry2021lowfrequency} in nature, and can easily handle heterogeneities and anisotropies of the head tissues conductivity profiles. Alternatively, the methods of the other category assume piecewise-homogeneity, which allows subdividing the overall head into $N$ compartments, representing separate biological tissues with distinct characteristics, such as, for example, the skin, the skull, the grey matter, and the white matter, and to approximate the above mentioned electrical properties as constant in each compartment \citep{darbas2019review}. By following this approach, there is no need for discretizing the full volumetric domain. Instead, a boundary element method (BEM) for the discretization of the integral formulation applied on the boundaries of the compartments only can be employed \citep{kybic2005common}, allowing the reduction of the dimensionality of the problem by one.
Moreover, boundary integral formulations applied to piecewise homogeneous head models can also handle anisotropies, as shown in \citep{pillain2019handling}.
The spherical head model (\Cref{subfig-2:sphere_model}), that is, the representation of the head as the union of constant conductivity nested layers with spherical boundaries, for which the analytic solution of the forward problem is available \citep{demunck1988potential,zhang1995fast}, and the realistic head model (\Cref{subfig-3:mri_model}), where the compartments in which the physical parameters are assumed constant are of general shape, are two examples of head models falling into this last category. 

An historical overview of the most commonly employed boundary element schemes employed to solve the forward EEG problem is available in \citep{kybic2005common}. The single- and double-layer approaches summarized there are based on a partial exploitation of the representation theorem \citep{nedelec2001acoustic, steinbach2008numerical}; the symmetric formulation instead, presented in \citep{kybic2005common}, follows from a clever full use of the representation theorem and represents nowadays one of the most favorable choices for the numerical solution of the forward problem \citep{darbas2019review, vorwerk2012comparison}. 
The symmetric formulation has been shown to provide a higher level of accuracy in its results when compared to the single- and double-layer approaches, especially for shallow sources approaching the boundaries of the compartments \citep{kybic2005common, vorwerk2012comparison}, which in turns leads to a higher reliability of the source localization algorithms based on it. This improvement comes at the cost of an higher computational complexity, that is, the linear system to be solved is approximately double in size with respect to the ones obtained from the single- and double-layer formulations. The increased cost is partially mitigated since the resulting interaction matrix becomes increasingly sparse when the number of compartments of the head model employed increases. Differently from the single- and double-layer formulations, which are second-kind integral equations, the symmetric formulation is of the first kind. As a consequence, its discretization results in a linear system whose conditioning worsens when the number of unknowns increases; this causes an increase of the solution time, a degradation of the solution accuracy, and, in some cases, prevents convergence altogether. This limitation actually prevents the application of the unpreconditioned symmetric formulation on complex realistic structures and opens the way of research toward efficient and effective preconditioning strategies, resulting in boundary integral formulations for the solution of the EEG forward problem that provide the same accuracy as the symmetric formulation and are immune to numerical instabilities.

In the last years, preconditioning strategies based on the Calder\'{o}n identities have gained in popularity in the computational electromagnetics community \citep{andriulli2008multiplicative}. They yield spectral equivalents of the inverse of the integral operators considered that are capable of preconditioning the formulation with respect to most sources of ill-conditioning, for example, denser discretization or lower frequencies. Calder\'{o}n preconditioning has been successfully applied to full-wave vectorial electromagnetic problems, such as the scattering from metallic \citep{andriulli2013wellconditioned,bagci2009calderon} and penetrable objects \citep{beghein2012calderon}, as well as to scalar, acoustic problems \citep{vantwout2021benchmarking, steinbach1998construction}. A Calder\'{o}n preconditioning strategy for the EEG symmetric formulation has also been proposed in \citep{ortizg.2018calderon}. The Calder\'{o}n preconditioning approach yields well-conditioned formulations in all these cases, at the cost of building a dual form of the BEM interaction matrix.
Most of these schemes require the evaluation of dense electromagnetic operators on dual basis functions defined on the dual mesh, constructed for example via a barycentric refinement of the primal mesh, which leads to a non-negligible computational burden. Recently, a refinement-free Calder\'{o}n preconditioning strategy for the electric field integral equation (EFIE) has been proposed \citep{adrian2019refinementfree}, which leverages suitably modified graph Laplacians.

In this work, we propose a Calder\'{o}n-like preconditioning for the symmetric formulation without resorting to integral operators discretized on the dual mesh, thereby reducing the computational cost compared with standard Calderón schemes such as \citep{ortizg.2018calderon}. The proposed formulation gives rise to a well-conditioned, symmetric, positive definite system of linear equations, amenable to fast iterative solvers, which remains stable under the different sources of ill-conditioning affecting the solution time of the non-preconditioned formulation.
We obtain this by leveraging the Laplace-Beltrami operator in our preconditioning scheme, which can cheaply be discretized on the primal and on the dual mesh, and suitable applying it to our multiplicative preconditioning scheme. In contrast to  \citep{adrian2019refinementfree}, special care needs to be taken in our theoretical framework given that the underlying operator of the symmetric formulation is block-structured and that the final formulation is required not only to be immune to the dense-discretization, but also to the high-contrast breakdown \citep{ortizg.2018calderon}.
Moreover, a consistent deflation strategy is introduced to avoid the nullspace of the symmetric formulation and the Laplace-Beltrami operator.
Numerical results demonstrate the effectiveness of our scheme, both for canonical and realistic head models.
Preliminary results, devoid of the theoretical apparatus presented here, have been presented in the conference contribution \citep{giunzioni2022new}.

This paper is organized as follows: in \Cref{sec:background} we review the forward EEG problem and the symmetric formulation for its numerical solution, providing the background and notation necessary for the following developments. \Cref{sec:conditionAnalysis} focuses on the numerical analysis of the two main sources of ill-conditioning plaguing the symmetric formulation. The novel preconditioning strategy is presented in \Cref{sec:formulation} following a step-by-step process: after the introduction of a proof-of-concept preconditioning example, we first propose a deflation strategy to obtain a non-singular symmetric formulation, before finally outlining the proposed well-posed, well-conditioned formulation, whose favorable stability properties are proved in \Cref{sec:well_cond}. Finally, \Cref{sec:result} complements the theoretical analysis with a numerical study of the formulations under test, to illustrate the effectiveness of the preconditioning. This analysis is applied to both canonical spherical models, for which a comparison with analytic solutions is possible, and realistic models generated from magnetic resonance imaging (MRI) data. Moreover, source localization algorithms relying on the proposed formulation have been applied on the realistic head model, validating the use of this technology for neuroimaging purposes in biomedical applications.

\section{Background and notation}
\label{sec:background}

In this section, we review the symmetric formulation for the solution of the forward EEG problem and set the notation for the elements needed to introduce the proposed formulation.

\subsection{The forward EEG problem}
The neural activity (i.e., the simultaneous activation of neurons situated in a region of the cortex) can be modeled by a primary current distribution, usually approximated as a combination of point dipoles \citep{demunck1988mathematical}. The forward EEG problem consists in determining the potential induced by this current at the scalp.

Mathematically, the problem is described by Poisson's equation, which is obtained from the Maxwell's system in its quasi-static approximation \citep{sarvas1987basic}.
This approximation is justified by the low frequency of the neural signals (mainly in the order of \SIrange{1}{100}{\Hz} \citep{darbas2019review}). Poisson's equation
\begin{equation}
\nabla \cdot \big(\sigma \nabla V\big) = \nabla \cdot \veg{j}
\end{equation}
describes the relation between the dipolar current inside the brain $\veg j$ (i.e., the source term of the problem) and the resulting potential distribution, $V$. 
In the forward problem, we are interested in solving the problem for spatial points $\veg r \in \Gamma \coloneqq \uppartial\Omega$, where $\Omega \subset \mathbb{R}^3$ is an open set modeling the head, and
the conductivity distribution $\sigma(\veg r)$ of the biological tissues is considered known.
Since the conductivity outside the head is zero, no current flows out of the head leading to the Neumann boundary condition
\begin{equation}
\sigma \uppartial_{\veg{n}} V = 0 \quad \text{for} \,\, \veg r \in \Gamma\,,
\end{equation}
where $\veg{n}$ is the outward unit normal vector field characterizing the boundary $\Gamma$ of the head domain.

\subsection{The symmetric formulation}
The symmetric formulation, which can be numerically solved in the BEM framework, relies on the piecewise-homogeneity assumption mentioned above, that is, we model the head as a set of non-overlapping, homogeneous compartments, representing distinct biological tissues.
Mathematically, we describe this by introducing
a set of $N$ nested, open subsets of $\Omega$, denoted as $\{\Omega_i\}_{i=1}^N$, with smooth boundaries, such that $\cup_{i=1}^{N}\bar{\Omega}_i = \bar{\Omega}$ and $\Omega_i \cap \Omega_{j\ne i} = \emptyset$.
Furthermore, $\Gamma_i$ denotes the boundary defined by $\Gamma_i \coloneqq \bar{\Omega_i} \cap \bar{\Omega}_{i+1}$, on which $\veg n_i$ is the unit-length normal vector directed towards $\Omega_{i+1}$, and $\Omega_{N+1}$ is the exterior of $\Omega$, that is, $\Omega_{N+1} \coloneqq \mathbb{R}^3\backslash\bar{\Omega}$. The notation employed is represented in \Cref{subfig-1:nested}.
In this work, we will assume isotropic conductivity modelled by the piecewise constant function
\begin{equation}
\sigma(\veg r) = \sum_{i=1}^{N+1} \sigma_i\, \text{\textit{1}}_{\Omega_i}(\veg r)\,,
\end{equation}
where $\sigma_i$ is scalar and $\text{\textit{1}}_{\Omega_i}(\veg r)$ is the indicator function of $\Omega_i$
\begin{equation}
\text{\textit{1}}_{\Omega_i}(\veg r) \coloneqq 
\begin{cases}
1\quad & \text{if $\veg r \in \Omega_i$}\, ,\\
0 \quad & \text{otherwise.}\\
\end{cases}
\end{equation}
Under these conditions, the forward EEG problem can be rewritten as
\begin{align}
\sigma_i \Deltaup V &= \nabla \cdot \veg{j} \quad \text{in} \,\, \Omega_i \, ,
\label{eqn:Poisson_boundary} \\
[V]_{\Gamma_i} &= 0 \quad\quad\,\,\, \forall i \le N \label{eqn:Poisson_bc1} \, ,\\
[\sigma \uppartial_{{\veg n}_i} V]_{\Gamma_i} &= 0 \quad\quad\,\,\, \forall i \le N , \label{eqn:Poisson_bc2}
\end{align}
where the symbol $[\cdot]_{\Gamma_i}$ denotes the jump of a function at the interface $\Gamma_i$, as defined in \citep{kybic2005common}.
Boundary conditions \eqref{eqn:Poisson_bc1} and \eqref{eqn:Poisson_bc2} enforce the continuity of the potential and of the current across compartments.

The problem described by \eqref{eqn:Poisson_boundary}-\eqref{eqn:Poisson_bc2} can be recast as an integral equation.
One such integral equation is the symmetric formulation, which is derived by exploiting the representation theorem \citep[Theorem~3.1.1]{nedelec2001acoustic} and expressing the complete solution $V$ as the sum of an homogeneous solution accounting for the source term and an harmonic function  chosen to satisfy the boundary conditions \eqref{eqn:Poisson_bc1} and \eqref{eqn:Poisson_bc2}.
The homogeneous solution in free space can be obtained by application of the potential theory.
Let 
\begin{equation}
G(\veg r, \veg r') \coloneqq  \frac{1}{4\uppi |\veg r-\veg r'|}
\end{equation}
be the Green's function associated with the Laplace equation \citep{felsen1994radiation} \citep[Equation~5.7]{steinbach2008numerical}, then $v(\veg r) = -\int \left(\nabla \cdot \veg j(\veg r)\right)\, G(\veg r, \veg r') \dd\veg r'$ satisfies $\Deltaup v = \nabla \cdot \veg j$ for all $\veg r \in \R^3$.
Following \cite{kybic2005common}, we define then the piecewise source function $f_{\Omega_i}(\veg r) \coloneqq  \left(\nabla \cdot \veg j(\veg r)\right) \cdot\text{\textit{1}}_{\Omega_i}(\veg r)$ and introduce
\begin{equation}
v_{\Omega_i}(\veg r) \coloneqq  -\int_{\Omega_i} f_{\Omega_i}(\veg r) \, G(\veg r,\veg r') d\veg r'\,.
\end{equation}

In the symmetric formulation, the harmonic function is constructed as
\begin{equation}
u_{\Omega_i} \coloneqq  \begin{cases}
V-\frac{v_{\Omega_i}}{\sigma_i} & \text{in $\Omega_i$}\, ,\\
-\frac{v_{\Omega_i}}{\sigma_i} & \text{elsewhere,}
\end{cases}
\end{equation}
from which a system of $2N$ integral equations can be derived (see \citep{kybic2005common} for further details), reading
\begin{equation}
\begin{cases}
(\uppartial_{\veg n} v_{\Omega_{i+1}})_{\Gamma_i}-(\uppartial_{\veg n} v_{\Omega_i})_{\Gamma_i} &= -\op{D}^*_{i,i-1}p_{i-1}+2\op{D}^*_{ii}p_i-\op{D}^*_{i,i+1}p_{i+1}
 \\
&\quad\quad+\sigma_i \op{N}_{i,i-1}V_{i-1}-(\sigma_i+\sigma_{i+1})\op{N}_{ii}V_i+\sigma_{i+1}\op{N}_{i,i+1}V_{i+1} 
\\
\sigma_{i+1}^{-1}(v_{\Omega_{i+1}})_{\Gamma_i}-\sigma_{i}^{-1}(v_{\Omega_i})_{\Gamma_i} &= \mathcal{D}_{i,i-1}V_{i-1}-2\mathcal{D}_{ii}V_i+\mathcal{D}_{i,i+1}V_{i+1}
 \\
&\quad\quad-\sigma_{i}^{-1}\mathcal{S}_{i,i-1}p_{i-1}+(\sigma_i^{-1}+\sigma_{i+1}^{-1})\mathcal{S}_{ii}p_i-\sigma_{i+1}^{-1}\mathcal{S}_{i,i+1}p_{i+1}.
\end{cases}
\label{eqn:symmformsyst}
\end{equation}
The unknowns of the system are $V_i \coloneqq (V)_{\Gamma_i}$, denoting the restriction of $V$ to $\Gamma_i$, and $p_i \coloneqq \sigma_i[   \,\uppartial_{{\veg n}_i} V]_{\Gamma_i}$.
The integral operators involved in equations \eqref{eqn:symmformsyst} are defined as
\begin{align}
&\mathcal{S}_{ij}: H^{-1/2}(\Gamma_i) \rightarrow  H^{1/2}(\Gamma_j), \quad\,\,\,\,\mathcal{S}_{ij}\psi(\veg{r}) \coloneqq  \int_{\Gamma_j}G(\veg{r}- \bm{r}') \psi(\veg{r}') \dd S(\veg{r}') 
\label{eqn:Sop}\\
&\mathcal{D}_{ij}: H^{1/2}(\Gamma_i) \rightarrow  H^{1/2}(\Gamma_j), \quad\quad \mathcal{D}_{ij}\phi(\veg{r}) \coloneqq  \mathrm{p.v.} \int_{\Gamma_j}\uppartial_{\veg{n}'_i} G(\veg{r}- \veg{r}') \phi(\veg{r}') \dd S(\veg{r}') 
\label{eqn:Dop}\\
&\mathcal{D}^*_{ij}: H^{-1/2}(\Gamma_i) \rightarrow  H^{-1/2}(\Gamma_j), \quad\,
\mathcal{D}_{ij}^*\psi(\veg{r}) \coloneqq  \mathrm{p.v.} \int_{\Gamma_j}\uppartial_{\veg{n}_j} G(\veg{r}- \veg{r}') \psi(\veg{r}') \dd S(\veg{r}') 
\label{eqn:Dsop}\\
&\mathcal{N}_{ij}: H^{1/2}(\Gamma_i) \rightarrow  H^{-1/2}(\Gamma_j), \quad\,\,\,\,
\mathcal{N}_{ij}\phi(\veg{r}) \coloneqq  \mathrm{f.p.} \int_{\Gamma_j}\uppartial_{\veg{n}_j}\uppartial_{\veg{n}'_i}G(\veg{r}- \veg{r}') \phi(\veg{r}')\dd S(\veg{r}')
\label{eqn:Nop}
\end{align}
and are respectively named single-layer, double-layer, adjoint double-layer, and hypersingular operators.
The definition of the Sobolev spaces $H^s$, $s \in \{-1/2, 1/2\}$, in the mapping properties above can be found in \citep{steinbach2008numerical, sauter2011boundary}. 
In the equations above, $\text{p.v.}$ and $\text{f.p.}$ indicate the Cauchy principal value and the Hadamard finite part.
The subscript $_{ij}$ denotes that these operators act on a function defined on $\Gamma_i$ and yield a function on $\Gamma_j$; if this subscript is omitted, $i = j$ is implicitly assumed.

\subsection{Discretization of the symmetric formulation}
\label{sec:symmform_discr}
We employ the BEM to discretize and numerically solve the system of equations \eqref{eqn:symmformsyst}.
To this end, a mesher triangulates the surfaces $\Gamma_i$ resulting in a set of $N$ meshes $\Gamma_{h,i}$, $i=1,\dots,N$.
In the following, the $_h$ subscript will be omitted in the case its meaning is already clear from the context.
Each $\Gamma_{h,i}$ is composed of $N_{C,i}$ triangular cells, $\{c_{i,n}\}_{n=1}^{N_{C,i}}$, of area $A_{i,n}$, and $N_{V,i}$ vertices, $\{v_{i,m}\}_{m=1}^{N_{V,1}}$, and is characterized by the mesh refinement parameter $h_i$ that is defined as the average length of the edges of $\Gamma_{h,i}$.
Over each $\Gamma_{h,i}$, we define a set of piecewise constant $\{\pi_{i,n}\}_{n=1}^{N_{C,i}}$ patch functions and a set of piecewise linear $\{\lambda_{i,m}\}_{m=1}^{N_{V,i}}$ pyramid functions as
\begin{equation}
\pi_{i,n}(\veg r) \coloneqq  \begin{cases}
1 & \text{for $\veg r \in c_{i,n}$,} \\
0 & \text{elsewhere,}
\end{cases} \quad \text{ and } \quad 
\lambda_{i,m}(\veg r) \coloneqq  \begin{cases}
1 & \text{for $\veg r = v_{i,m}$,} \\
0 & \text{for $\veg r = v_{i,p\ne m}$,} \\
\text{linear} & \text{elsewhere.}
\end{cases}
\label{eqn:patchpyrdef}
\end{equation}
For the boundary element spaces spanned by these functions, $X_{\pi_i} \coloneq \text{span}\{\pi_{i,n}\}_{n=1}^{N_{C,i}}$ and $X_{\lambda_i} \coloneq \text{span}\{\lambda_{i,m}\}_{m=1}^{N_{V,i}}$, we have $X_{\pi_i} \subset H^{-1/2}(\Gamma_{h,i})$ , $X_{\lambda_i} \subset H^{1/2}(\Gamma_{h,i})$ \citep{steinbach2008numerical}.

Following the standard Galerkin procedure, we expand the unknowns of the system in \eqref{eqn:symmformsyst} as
\begin{equation}
V_i \approx \sum_{m=1}^{N_{V,i}}{l}_{i,m} \lambda_{i,m}
\quad\text{and}\quad
p_i \approx \sum_{n=1}^{N_{C,i}}{p}_{i,n} \pi_{i,n}
\label{eqn:pexp}
\end{equation}
and we test with pyramid and patch functions resulting in the linear system of equations
\begin{equation}
\mat Z \begin{pmatrix}
\vec l \\ \vec p
\end{pmatrix} = \begin{pmatrix}
\vec b \\ \vec c
\end{pmatrix}\,.
\label{eqn:symmform}
\end{equation}
The block matrix $\mat Z$ can be represented as
\begin{equation}
\mat Z = \begin{pmatrix}
\mat N &  \mat D^* \\
\mat D & \mat S
\end{pmatrix}\,,
\label{eqn:Zmatrix}
\end{equation}
where each block is another block matrix composed out of $N^2$ blocks, whose position inside $\{ \mat N,\mat D^*,\mat D,\mat S \}$ will be identified by a block row index and a block column index. By denoting the blocks of $\{ \mat N,\mat D^*,\mat D,\mat S \}$ in block row $x$ and block column $y$, for $x,y = 1,...,N$, with the superscript $^{xy}$, their definitions are
\begin{align}
\mat N^{xy} &\coloneqq \begin{cases}
(\sigma_x+\sigma_{x+1})\mat N_{xy} & \text{if $x=y$} \\
-\sigma_y \mat N_{xy} & \text{if $x=y-1$} \\
-\sigma_x \mat N_{xy} & \text{if $x=y+1$} \\
0\cdot \mat N_{xy} & \text{otherwise} \\
\end{cases}\,,\quad\quad\quad\quad
\mat D^{*xy} \coloneqq \begin{cases}
-2\mat D^*_{xy} & \text{if $x=y$} \\
\mat D^*_{xy} & \text{if $x=y\pm1$} \\
0\cdot \mat D^*_{xy} & \text{otherwise} \\
\end{cases}\,,\nonumber\\
\mat D^{xy} &\coloneqq \begin{cases}
-2\mat D_{xy} & \text{if $x=y$} \\
\mat D_{xy} & \text{if $x=y\pm1$} \\
0\cdot \mat D_{xy} & \text{otherwise} \\
\end{cases}\,,\quad\quad\quad\quad\quad\quad\quad\,\,\,\,\,\,\,
\mat S^{xy} \coloneqq \begin{cases}
(\sigma_x^{-1}+\sigma_{x+1}^{-1})\mat S_{xy} & \text{if $x=y$} \\
-\sigma_y^{-1} \mat S_{xy} & \text{if $x=y-1$} \\
-\sigma_x^{-1} \mat S_{xy} & \text{if $x=y+1$} \\
0\cdot \mat S_{xy} & \text{otherwise}
\end{cases}\,.\nonumber
\end{align}
The matrices introduced in the previous equations are defined as
\begin{align}
(\mat N_{xy})_{mn} &\coloneqq \left( \lambda_{x,m}, \mathcal{N}_{xy} \lambda_{y, n} \vphantom{\mathcal{D}^*_{xy}} \right)_{L^2(\Gamma_x)}\,,\\
(\mat D^*_{xy})_{mn} &\coloneqq \left( \lambda_{x,m}, \mathcal{D}^*_{xy} \pi_{y, n} \right)_{L^2(\Gamma_x)}\,,\\
(\mat D_{xy})_{mn} &\coloneqq \left( \pi_{x,m}, \mathcal{D}_{xy} \lambda_{y, n} \vphantom{\mathcal{D}^*_{xy}}\right)_{L^2(\Gamma_x)}\,,\\
(\mat S_{xy})_{mn} &\coloneqq \left( \pi_{x,m}, \mathcal{S}_{xy} \pi_{y, n} \vphantom{\mathcal{D}^*_{xy}}\right)_{L^2(\Gamma_x)}\,.
\end{align}
The blocks of the right-hand-side (RHS) vector in equation \eqref{eqn:symmform} can be written as
\begin{equation}
\vec b = \begin{pmatrix}
\vec b_1 \\ \vec b_2 \\ \vdots \\ \vec b_N
\end{pmatrix}\,, \quad 
\vec c = \begin{pmatrix}
\vec c_1 \\ \vec c_2 \\ \vdots \\ \vec c_N
\end{pmatrix},
\end{equation}
where
\begin{align}
(\vec b_i)_m &\coloneqq \left( \lambda_{i,m},(  \uppartial_{\veg n} v_{\Omega_{i+1}}-\uppartial_{\veg n} v_{\Omega_i})\right)_{L^2(\Gamma_i)}\,,\\
(\vec c_i)_n &\coloneqq \left( \pi_{i,n},(  \sigma_{i+1}^{-1}v_{\Omega_{i+1}}-\sigma_{i}^{-1}v_{\Omega_i})\right)_{L^2(\Gamma_i)}.
\end{align}
The blocks of the unknown vector in equation \eqref{eqn:symmform} can be written as
\begin{equation}
\vec l = \begin{pmatrix}
\vec l_1 \\ \vec l_2 \\ \vdots \\ \vec l_N
\end{pmatrix}\,, \quad 
\vec p = \begin{pmatrix}
\vec p_1 \\ \vec p_2 \\ \vdots \\ \vec p_N
\end{pmatrix},
\end{equation}
whose elements are simply $(\vec l_i)_m \coloneqq l_{i,m}$ and $(\vec p_i)_n \coloneqq p_{i,n}$.
The elements of $\vec l_i$ are the voltages at the vertices of $\Gamma_{h,i}$ due to the interpolatory nature of the basis functions.
As a final remark, the last block-row and the last block-column of the linear system in \eqref{eqn:symmform} need to be eliminated \citep{kybic2005common}, as the exterior conductivity $\sigma_{N+1}$ is null.

\subsection{Operators and dual basis functions needed for the new formulation}
\label{sec:further_notation}
This section establishes further operators and basis functions needed for the new formulation.
First, we barycentrically refine $\Gamma_{h,i}$ resulting in $\bar{\Gamma}_{h,i}$. Using $\bar{\Gamma}_{h,i}$ we obtain the dual mesh $\tilde{\Gamma}_{h,i}$, where the vertices of $\Gamma_{h,i}$ have become cells and the cells of $\Gamma_{h,i}$ have become vertices (see \Cref{fig:mesh}). 
On $\tilde{\Gamma}_{h,i}$, we define dual pyramid basis functions $\tilde{\lambda}_{i,n}$, which are attached to the dual vertices $\tilde{v}_{i,n}$, defined by
\begin{equation}
\tilde{\lambda}_{i,n}(\veg r) \coloneqq  \sum_{x=1}^7 \frac{1}{\mathit{NoC}(\bar{v}_{i,n,x})} \bar{\lambda}_{i,n,x}(\veg r),
\end{equation}
where $\bar{\lambda}_{i,n,x}$ is the pyramid function with domain on the barycentrically refined mesh $\bar{\Gamma}_{h,i}$, attached to the $x$th vertex of $\bar{\Gamma}_{h,i}$, $\{\bar{v}_{i,n,x}\}_{x=1}^7$, lying on the primal cell $c_{i,n}$, and the function $\mathit{NoC}(\bar{v}_{i,n,x})$ gives the number of primal cells connected to the vertex $\bar{v}_{i,n,x}$. A graphical representation of the dual pyramid basis function is shown in \Cref{subfig-2:dualPyr}.

\begin{figure}
\subfloat[\label{subfig-1:mesh_primal}]{%
  \includegraphics[width=0.3
\columnwidth]{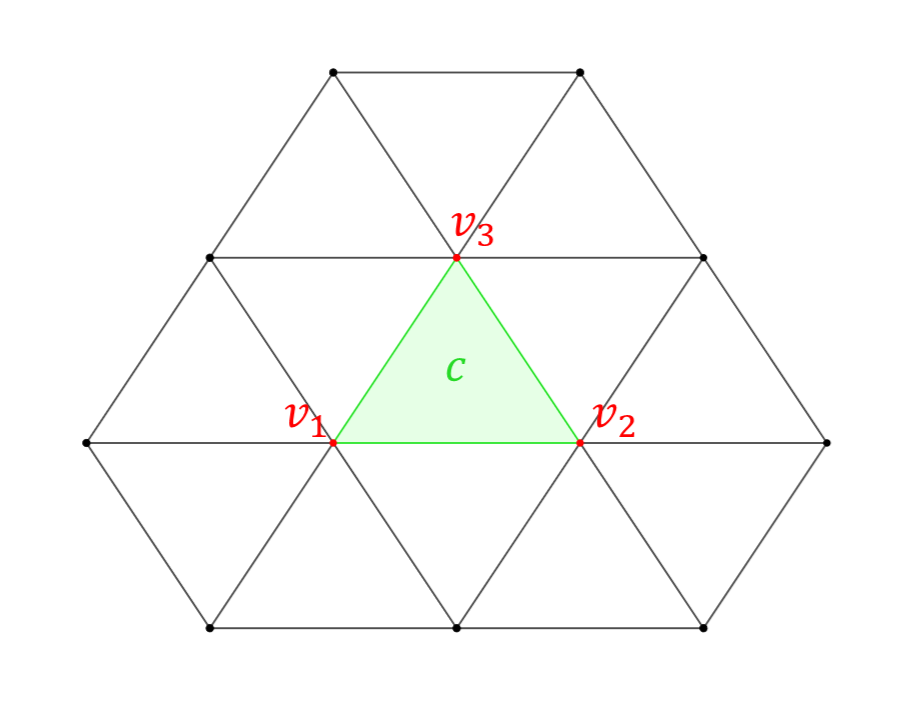}
}
\hfill
\subfloat[\label{subfig-2:mesh_ref}]{%
  \includegraphics[width=0.3\columnwidth]{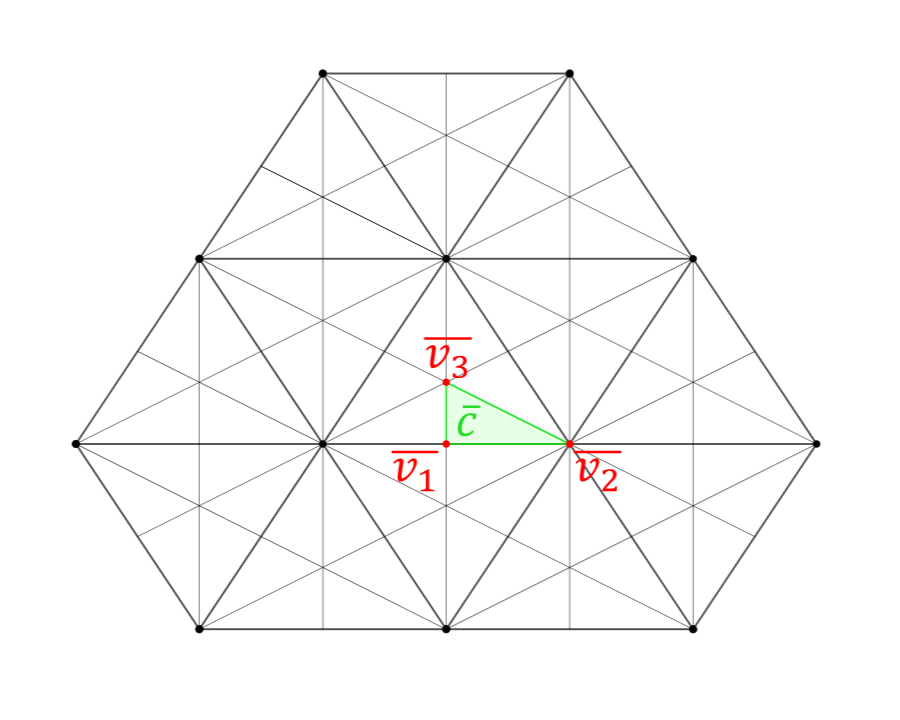}
}
\hfill
\subfloat[\label{subfig-3:mesh_dual}]{%
  \includegraphics[width=0.3\columnwidth]{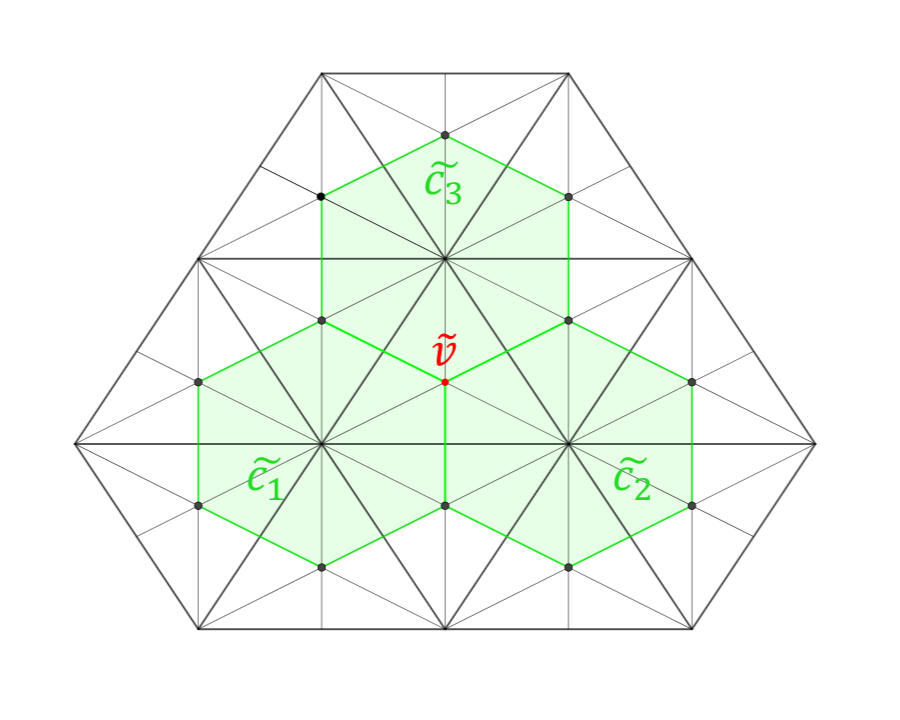}
}
\caption{(a) Primal, (b) barycentrically refined, and (c) dual mesh, denoted respectively as $\Gamma_{h,i}$, $\bar{\Gamma}_{h,i}$, and $\tilde{\Gamma}_{h,i}$. The vertices are in red, the cells are in green.}
\label{fig:mesh}
\end{figure}

\begin{figure}
\subfloat[\label{subfig-1:primPyr}]{%
  \includegraphics[width=0.3
\columnwidth]{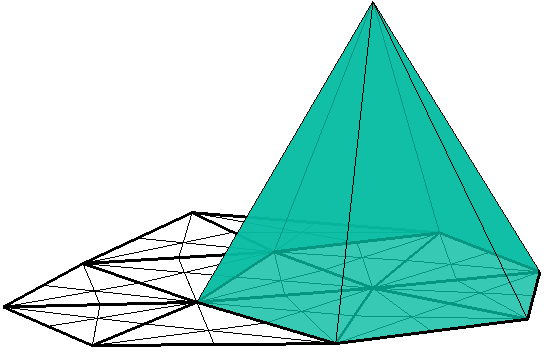}
}
\hfill
\subfloat[\label{subfig-2:dualPyr}]{%
  \includegraphics[width=0.3\columnwidth]{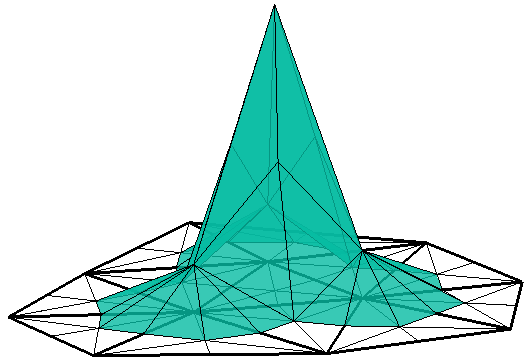}
}
\subfloat[\label{subfig-3:dualPyrDom}]{%
  \includegraphics[width=0.4\columnwidth]{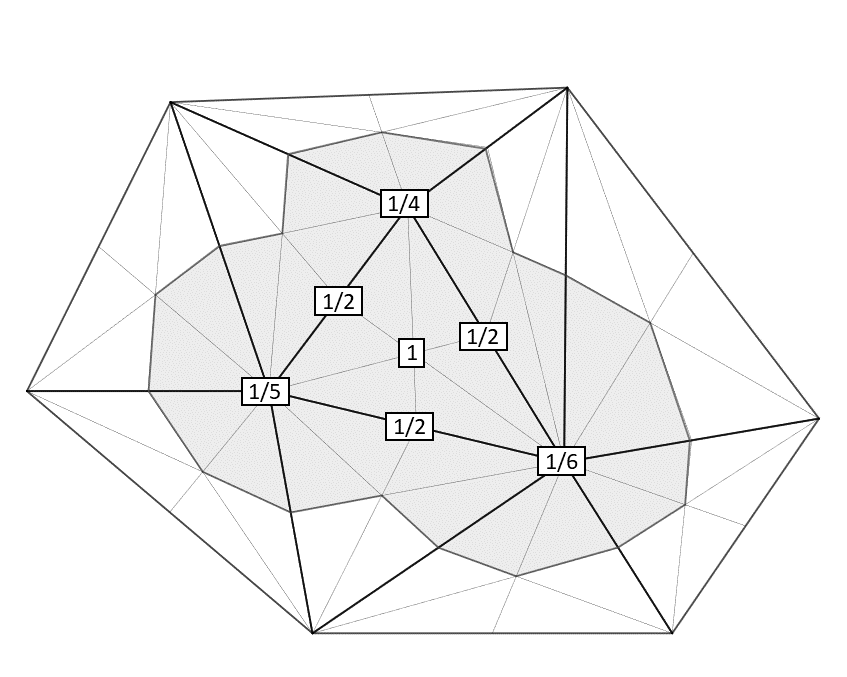}
}
\caption{(a) Primal and (b) dual pyramid basis function representation. (c) Dual pyramid basis function support: the numbers denote the value of the function at the point.}
\end{figure}

In the following, we will also need the identity $\mathcal{I}$ and the Laplace-Beltrami $\Deltaup_{\Gamma_i}$ operators. They can be discretized through their application to a given set of expansion functions and the testing with a set of test basis functions. For the identity operator, the outcome of this operation is usually named Gram matrix, denoted in this work by $\mat G_{i,fg}$, where $f,g \in \{\pi, \lambda, \tilde{\lambda}\}$ indicate the type of test and expansion functions employed in the discretization
\begin{equation}
(\mat G_{i,fg})_{mn} \coloneqq  \left( f_{i,m}(\veg r), \, g_{i,n} (\veg r) \right)_{L^2(\Gamma_i)}.
\end{equation}
The discretization of the Laplace-Beltrami operator by means of pyramid and dual pyramid functions leads to the matrices
\begin{align}
(\mat \Delta_i)_{mn} &\coloneqq  \left( \nabla_{\Gamma_i}\lambda_{i,m},\nabla_{\Gamma_i}\lambda_{i,n}  \right)_{L^2(\Gamma_i)} \, ,\\
(\tilde{\mat \Delta}_i)_{mn} &\coloneqq  \left( \nabla_{\Gamma_i}\tilde{\lambda}_{i,m},\nabla_{\Gamma_i}\tilde{\lambda}_{i,n}  \right)_{L^2(\Gamma_i)}\, ,
\end{align}
where $\nabla_{\Gamma_i}$ denotes the surface gradient operator \citep[Equation~4.200]{sauter2011boundary}.
The constant function along $\Gamma_i$, $\text{\textit{1}}(\veg r) \coloneqq 1$, is an eigensolution of the Laplace-Beltrami operator, that is $\Deltaup_{\Gamma_i} \text{\textit{1}} = 0$. Hence, the one-dimensional nullspace of the Laplace-Beltrami operator is spanned by the constant function, in symbols $\text{ker}\,\Deltaup_{\Gamma_i} = \text{span}\{\text{\textit{1}}\}$.
Due to our choice of basis and test functions, $\mat \Delta_i$ and $\tilde{\mat \Delta}_i$ have, likewise, a non-trivial nullspace given by $\bm{1}_{N_{V,i}}$ and $\bm{1}_{N_{C,i}}$, where $\bm{1}_n$ denotes the all-one vector in $\mathbb{R}^{n}$.

\section{Analysis of the conditioning of the symmetric formulation}
\label{sec:conditionAnalysis}
The ill-conditioning plaguing the symmetric formulation can be traced back to two distinct numerical effects: the dense-discretization breakdown (i.e., the increase of the condition number for $h \rightarrow 0$, where $h \coloneqq \min_i h_i$) and the high-contrast breakdown (i.e. the increase of the condition number for an increasing conductivity contrast between adjacent compartments).

The stability analysis proposed here aims at characterizing the discrete spectral behaviour of the formulation, in particular through the estimation of the condition number of $\mat Z$ \citep[Equation~2.1.25]{quarteroni2008numerical}, $\text{cond}(\mat Z)$, given by the ratio between the highest and the lowest singular values of $\mat Z$
which is known to impact, for most iterative solvers, both the speed of convergence and the achievable accuracy (in finite precision arithmetic) \citep{quarteroni2008numerical}.

\subsection{Dense-discretization behaviour}
\label{sec:condH}
In this section, we will study the stability properties of the discrete symmetric formulation when $h$ decreases. 
We base our analysis on the spherical harmonics decomposition of the operator; thus, we employ a spherical multi-compartment head model. We do not loose generality here, as realistic multi-compartment head models will have similar spectral properties since a smooth deformation of a geometry gives rise to a compact perturbation of the electrostatic operators under study \citep{sloan1992error, kolm2003quadruple}.
The assumption of a smooth deformation is uncritical as human heads typically do not have sharp edges and corners.

We consider an head model characterized by $N$ concentric, spherical boundaries $\{\Gamma_i\}_{i=1}^N$ with radius $\{R_i\}_{i=1}^N$. The operator under study 
\begin{equation}
\mathcal{Z} = \begin{pmatrix}
\mathcal{N}_b &  \mathcal{D}^*_b \\
\mathcal{D}_b & \mathcal{S}_b
\end{pmatrix}
\label{eqn:Zopmatrix}
\end{equation}
is a block integral operator, where each block operator is another block operator composed of $N^2$ blocks. By denoting the operator in block row $x$ and block column $y$ of $\{ \mathcal{N}_b,\mathcal{D}_b^*,\mathcal{D}_b,\mathcal{S}_b \}$, for $x,y = 1,...,N$, with the superscript $^{xy}$, we have
\begin{align}
\mathcal{N}^{xy}_b &\coloneqq  \begin{cases}
(\sigma_x+\sigma_{x+1})\mathcal{N}_{xy} & \text{if $x=y$} \\
-\sigma_y \mathcal{N}_{xy} & \text{if $x=y-1$} \\
-\sigma_x \mathcal{N}_{xy} & \text{if $x=y+1$} \\
0\cdot \mathcal{N}_{xy} & \text{otherwise} \\
\end{cases}\quad\quad\quad\quad
\mathcal{D}^{*xy}_b \coloneqq  \begin{cases}
-2\mathcal{D}^*_{xy} & \text{if $x=y$} \\
\mathcal{D}^*_{xy} & \text{if $x=y\pm1$} \\
0\cdot \mathcal{D}^*_{xy} & \text{otherwise} \\
\end{cases}\nonumber\\
\mathcal{D}^{xy}_b &\coloneqq  \begin{cases}
-2\mathcal{D}_{xy} & \text{if $x=y$} \\
\mathcal{D}_{xy} & \text{if $x=y\pm1$} \\
0\cdot \mathcal{D}_{xy} & \text{otherwise} \\
\end{cases}\quad\quad\quad\quad\quad\quad\quad\,\,\,\,\,\,\,
\mathcal{S}^{xy}_b \coloneqq  \begin{cases}
(\sigma_x^{-1}+\sigma_{x+1}^{-1})\mathcal{S}_{xy} & \text{if $x=y$} \\
-\sigma_y^{-1} \mathcal{S}_{xy} & \text{if $x=y-1$} \\
-\sigma_x^{-1} \mathcal{S}_{xy} & \text{if $x=y+1$} \\
0\cdot \mathcal{S}_{xy} & \text{otherwise}
\end{cases}.\nonumber
\end{align}
We begin our analysis by noting that due to the compactness of the double-layer operator $\mathcal{D}_{xy}$ \eqref{eqn:Dop} and its adjoint counterpart $\mathcal{D}^*_{xy}$ \eqref{eqn:Dsop} on smooth domains \citep{costabel1985direct}, the blocks $\mathcal{D}_b$, $\mathcal{D}^*_b$ are block operators, where each block is a compact operator; hence, $\mathcal{D}_b$, $\mathcal{D}^*_b$  are compact (see \ref{sec:proof_blockcompact} for a proof). Moreover, the single-layer $\mathcal{S}_{i,j\ne i}$ \eqref{eqn:Sop} and the hypersingular operators $\mathcal{N}_{i,j\ne i}$ \eqref{eqn:Nop} are compact as their kernels are continuous provided the analyticity of the Green's function evaluated in $\veg r$ far enough from $\veg r'$ \citep[Chapter~5.1.3]{sauter2011boundary}, \citep[Theorem~1.10]{colton2013integral}. Finally, the operator $\mathcal{Z}$ can be decomposed into the sum of a block operator involving only its diagonal blocks and another block operator, $\mathcal{K}_{\mathcal{Z}}$, containing the off-diagonal, compact contributions,
\begin{align}
\mathcal{Z} = \diag\big( &(\sigma_1+\sigma_2) \mathcal{N}_{11}, (\sigma_2+\sigma_3) \mathcal{N}_{22} ,...,(\sigma_N+\sigma_{N+1}) \mathcal{N}_{NN},\nonumber\\
&(\sigma_1^{-1}+\sigma_2^{-1}) \mathcal{S}_{11}, (\sigma_2^{-1}+\sigma_3^{-1}) \mathcal{S}_{22},...,(\sigma_N^{-1}+\sigma_{N+1}^{-1}) \mathcal{S}_{NN}\big)+\mathcal{K}_{\mathcal{Z}}\,.
\label{eqn:Zop_diag_offdiag}
\end{align}

The contribution $\mathcal{K}_{\mathcal{Z}}$, a block operator with compact blocks, is compact, as proved in \ref{sec:proof_blockcompact}.
Since its eigenvalues accumulate at zero when increasing the number of degrees of freedom of the problem \citep[Theorem~1.34]{colton2013integral}, the dominant spectral properties of $\mathcal{Z}$ are determined by the ones of the first, non-compact, operator in the summation \eqref{eqn:Zop_diag_offdiag}.
Therefore, we have to study the eigenvalues of the first term in \eqref{eqn:Zop_diag_offdiag}, given by the union of the sets of eigenvalues of each diagonal block.
We obtain those eigenvalues by noting that the spherical harmonic function $Y_{lm}$ of degree $l$ and order $m$ with $l \ge 0$ and $|m|<l$ (for a formal definition, see \citep[Definition~14.30.1]{olver2010nist}) is an eigenfunction of the single-layer and the hypersingular operator \citep{hsiao1994error}
\begin{align}
\mathcal{S}_{ii}Y_{lm}&=\uplambda_{l,S_{ii}} Y_{lm}\, ,\\
\mathcal{N}_{ii}Y_{lm}&=\uplambda_{l,N_{ii}}Y_{lm}\,,
\end{align}
where the eigenvalues of the previous equations are
\begin{align}
\uplambda_{l,\mathcal{S}_{ii}}&=\frac{R_i}{2l+1}\,,\quad\,\,\,\quad\,\,\,\,\,\quad \text{with multiplicity}\,\, 2l+1\\
\uplambda_{l,\mathcal{N}_{ii}}&=-\frac{l(l+1)}{R_i(2l+1)}\,,\quad\quad \text{with multiplicity}\,\, 2l+1\,.
\end{align}
Due to the definiteness property of the self-adjoint operators $\mathcal{S}_{ii}$, $\mathcal{N}_{ii}$, their singular values are simply obtained as $\upsigma_{l,\mathcal{S}_{ii}} = \uplambda_{l,\mathcal{S}_{ii}}$ and $\upsigma_{l,\mathcal{N}_{ii}} = -\uplambda_{l,\mathcal{N}_{ii}}$.

The resulting set of eigenvalues of $(\mathcal{Z}-\mathcal{K}_{\mathcal{Z}})$ reads
\begin{align}
\text{eig}(\mathcal{Z}-\mathcal{K}_{\mathcal{Z}}) =
&\{(\sigma_1+\sigma_2)\,\uplambda_{l,N_{11}},\,
(\sigma_2+\sigma_3)\,\uplambda_{l,N_{22}},\,...,\,
(\sigma_l+\sigma_{l+1})\,\uplambda_{l,N_{NN}}\}\,\cup\nonumber\\
&\{(\sigma_1^{-1}+\sigma_2^{-1})\,\uplambda_{l,S_{11}},\,
(\sigma_2^{-1}+\sigma_3^{-1})\,\uplambda_{l,S_{22}},\,...,\,
(\sigma_{N-1}^{-1}+\sigma_{N}^{-1})\,\uplambda_{l,S_{N-1,N-1}}\}\,.
\label{eqn:KmK_eig}
\end{align}
Hence, we can recognize different branches of eigenvalues of $(\mathcal{Z}-\mathcal{K}_{\mathcal{Z}})$: on one side, those associated to $\mathcal{N}_{ii}$, $(\sigma_i+\sigma_{i+1})\,\uplambda_{l,N_{ii}}$, diverging towards minus infinity with the order $l$ of the corresponding eigenfunction, and, on the other side, those associated to $\mathcal{S}_{ii}$, $(\sigma_i^{-1}+\sigma_{i+1}^{-1})\,\uplambda_{l,S_{ii}}$, converging toward $0$ for $l\rightarrow\infty$.

To obtain a statement on how the condition number grows in $h_i$, we note that the maximum degree $l$ supported by a mesh resolution $h_i$ behaves asymptotically as $l = \mathcal{O}(R_ih_i^{-1})$ \citep{adrian2021electromagnetic}. We then find
\begin{align}
\upsigma_{l,\mathcal{N}_{ii}} &= -\uplambda_{l,\mathcal{N}_{ii}} = \mathcal{O}(R_i^{-1}l) = \mathcal{O}(h_i^{-1})\,,\\
\upsigma_{l,\mathcal{S}_{ii}} &= \uplambda_{l,\mathcal{S}_{ii}} = \mathcal{O}(R_il^{-1}) = \mathcal{O}(h_i),
\end{align}
showing the ill-conditioned nature of the symmetric formulation and resulting in a condition number growth of its discretization by means of an $L^2$-orthonormal basis as $\mathcal{O}(h^{-2})$.

\subsection{High-contrast behaviour}
\label{sec:HCR_breakdown}

From the expression of the eigenvalues of $(\mathcal{Z}-\mathcal{K}_{\mathcal{Z}})$ on a spherical head model \eqref{eqn:KmK_eig}, another source of ill conditioning can be verified, related to the conductivity contrast between different compartments of the head domain $\Omega$.
We define the conductivity ratio between the adjacent compartments $\Omega_i$ and $\Omega_{i+1}$ as
\begin{equation}
\mathit{CR}_{i} \coloneqq  \frac{\max(\sigma_i, \sigma_{i+1})}{\min(\sigma_i, \sigma_{i+1})}\,,
\end{equation}
and can see that the condition number grows with increasing $\mathit{CR}_i$ by first considering that, asymptotically, $\mathit{CR}_{i}\rightarrow \infty$ corresponds either to the condition $(\sigma_i+\sigma_{i+1})\rightarrow \infty$ when $\max(\sigma_i, \sigma_{i+1}) \rightarrow \infty$, or to $(\sigma_i^{-1}+\sigma_{i+1}^{-1}) \rightarrow 0$ when $\min(\sigma_i, \sigma_{i+1}) \rightarrow 0$.
Then, the absolute difference between the branch of eigenvalues $(\sigma_i+\sigma_{i+1})\,\uplambda_{l,N_{ii}}$ and the branch $(\sigma_i^{-1}+\sigma_{i+1}^{-1})\,\uplambda_{l,S_{ii}}$ grows when increasing the parameter $\mathit{CR}_{i}$, leading to an increasing condition number. Therefore, we conclude that the instability of the symmetric formulation is worsened in presence of high-contrast models. The high-contrast behavior, together with the dense-discretization ill-conditioning outlined in the previous \Cref{sec:condH}, actually places a limit on the effective use of the symmetric formulation applied to realistic scenarios, where the head model is characterized by high-contrasts between the skull and the nearby tissues and the meshes are heavily refined, in order to capture small anatomical details since high condition numbers can lead to slowly or non-converging iterative solvers.

\section{The new formulation} \label{sec:formulation}
The preconditioning strategy proposed in this work makes the symmetric formulation immune to both the dense-discretization and the high-contrast breakdowns. Differently from standard Calder\'{o}n approaches, ours is obtained at continuous level through the product of three operators. To clarify the ideas, before introducing the complete formulation, we will consider the regularization of the single-layer operator only as a proof of concept.

\subsection{A proof of concept}
\label{sec:POC}
The single-layer operator $\mathcal{S}_{11}: H^{-1/2}(\Gamma_1) \rightarrow  H^{1/2}(\Gamma_1)$, where $\Gamma_1$ is smooth, is a pseudo-differential operator of order $-1$ \citep{steinbach1998construction}.  
From its boundedness and $H^{-1/2}-$ellipticity \citep[Equation~6.8]{ steinbach2008numerical}, it follows that, for any $f\in H^{1/2}(\Gamma_1)$, the solution of the equation $\mathcal{S}u=f$ is unique (Lax-Milgram theorem,  \citep[Theorem~3.4]{steinbach2008numerical}).
However, as can be seen from the spectral analysis in \Cref{sec:condH}, this integral equation is first-kind in nature and its discretization gives rise to an ill-conditioned system of linear equations.
As suggested in \citep{steinbach1998construction}, a preconditioning operator for $\mathcal{S}$ can be constructed by explicitly evaluating the inverse of the principal symbol of $\mathcal{S}$ and by finding its associated self-adjoint, elliptic preconditioning operator, of pseudodifferential order $1$. This strategy, better known as Calder\'{o}n preconditioning, %
has a theoretical basis in the Calder\'{o}n identity \citep[Equation~3.1.45]{nedelec2001acoustic}
\begin{equation}
-\mathcal{N}\mathcal{S} = \frac{\mathcal{I}}{4}-\mathcal{D}^{*2},
\end{equation}
stating that the product of the hypersingular and the single-layer operator gives rise to a second kind operator.

Other possibilities exist, however.
We consider for example the operator $\mathcal{S}_{11}\Deltaup_{\Gamma_1} \mathcal{S}_{11}$, where the Laplace-Beltrami operator $\Deltaup_{\Gamma_1}: H^{\alpha}(\Gamma_1) \rightarrow  H^{\alpha-2}(\Gamma_1)$ is a pseudo-differential operator of order $2$. The product $\mathcal{S}_{11}\Deltaup_{\Gamma_1}: H^{-1/2}(\Gamma_1) \rightarrow  H^{-3/2}(\Gamma_1)$ satisfies the condition of being of order $1$, so it could give rise to a suitable left preconditioner for $\mathcal{S}_{11}$.
The favorable conditioning properties of $\mathcal{S}_{11}\Deltaup_{\Gamma_1} \mathcal{S}_{11}$ have been shown in \citep{oneil2018secondkind}, where, by explicit expansion of the product, $\mathcal{S}_{11}\Deltaup_{\Gamma_1}\mathcal{S}_{11}$ has been proved to be a second-kind integral operator, with its spectrum accumulating at $1/4$,
\begin{equation}
\mathcal{S}_{11}\Deltaup_{\Gamma_1}\mathcal{S}_{11} = \frac{\mathcal{I}}{4} + \mathcal{K}_{\mathcal{S}_{11}\Deltaup_{\Gamma_1}\mathcal{S}_{11}}\, ,
\end{equation}
where $\mathcal{K}_{\mathcal{S}_{11}\Deltaup_{\Gamma_1}\mathcal{S}_{11}}$ is compact.
It follows that the matrix $\mat G_{1,\pi\pi}^{-1/2} \mat S_{11} \mat G_{1,\tilde{\lambda}\pi}^{-1} \tilde{\mat \Delta}_1 \mat G_{1,\tilde{\lambda}\pi}^{-1}\mat S_{11} \mat G_{1,\pi\pi}^{-1/2}$, discretizing the second kind operator $\mathcal{S}\Deltaup_{\Gamma_1}\mathcal{S}$ by means of an orthonormal set of basis functions, is well-conditioned up to its nullspace and its spectrum accumulates at the point $1/4$.

However, the discrete preconditioning presented above is in general not allowed, because of the one-dimensional nullspace of $\tilde{\mat \Delta}_1$, which is spanned by the all-one vector. Indeed, the application of a singular preconditioner to a non-singular system of linear equations causes a loss of information which prevents recovering the solution of the original problem. To overcome this issue, following the approach presented in \citep{steinbach1998construction}, we introduce the operator $\hat{{\Deltaup}}_{\Gamma_i}: H^{\alpha}(\Gamma_i) \rightarrow  H^{\alpha-2}(\Gamma_i)$ defined by the bilinear form
\begin{equation}
\left( v,\hat{\Deltaup}_{\Gamma_i}w \right)_{L^2(\Gamma_i)} := \left( \nabla_{\Gamma_i}w,\nabla_{\Gamma_i}v \right)_{L^2(\Gamma_i)} + \left( \text{\textit{1}},w \right)_{L^2(\Gamma_i)}\left( \text{\textit{1}},v \right)_{L^2(\Gamma_i)}.
\end{equation}
The operator $\hat{{\Deltaup}}_{\Gamma_i}$ is invertible. Therefore, the unique solution of the problem
\begin{equation}
\hat{\Deltaup}_{\Gamma_i}w = g
\end{equation}
is also a solution of the ill-posed problem
\begin{equation}
{\Deltaup}_{\Gamma_i}w = g
\end{equation}
if $g$ satisfies the solvability condition $\int_{\Gamma_i}g \dd S = 0$.
The discretizations of $\hat{\Deltaup}_{\Gamma_i}$ by means of pyramid functions defined on $\Gamma_i$,
\begin{equation}
\hat{{\mat\Delta}}_i \coloneqq  \mat\Delta_i+\mat G_{i,\lambda\lambda}^\T \bm{{1}}_{N_{V,i}} \bm{1}_{N_{V,i}}^\T\mat G_{i,\lambda\lambda}\, ,
\end{equation}
and the one by means of dual pyramid functions defined on $\tilde{\Gamma}_i$,
\begin{equation}
\hat{\tilde{\mat\Delta}}_i \coloneqq  \tilde{\mat\Delta}_i+\mat G_{i,\tilde{\lambda}\tilde{\lambda}}^\T \bm{1}_{N_{C,i}} \bm{1}_{N_{C,i}}^\T\mat G_{i,\tilde{\lambda}\tilde{\lambda}}\, ,
\end{equation}
are non-singular matrices.
As in the case of $\mathcal{S}_{11}\Deltaup_{\Gamma_1} \mathcal{S}_{11}$, it can be shown that $\mathcal{S}_{11}\hat{\Deltaup}_{\Gamma_1} \mathcal{S}_{11}$ is a second-kind integral operator accumulating at $1/4$. Indeed, by expanding its weak form, following the steps in \citep{oneil2018secondkind}, one obtains
\begin{align}
(\mathcal{S}_{11}\hat{\Deltaup}_{\Gamma_1} \mathcal{S}_{11}v,w)_{L^2(\Gamma_1)} = \frac{1}{4}(v,w)_{L^2(\Gamma_1)} + (\mathcal{K}_{\mathcal{S}_{11}\Deltaup_{\Gamma_1}\mathcal{S}_{11}} v,w)_{L^2(\Gamma_1)}-(\text{\textit{1}},\mathcal{S}_{11}v)_{L^2(\Gamma_1)}(\text{\textit{1}},w)_{L^2(\Gamma_1)}.
\end{align}
Since the term $(\text{\textit{1}},\mathcal{S}_{11}v)_{L^2(\Gamma_1)}(\text{\textit{1}},w)_{L^2(\Gamma_1)}$ represents the bilinear form of a separable operator with finite dimensional range
, $\mathcal{S}_{11}\hat{\Deltaup}_{\Gamma_1} \mathcal{S}_{11}$ is a second-kind operator and its strong form can be written as
\begin{equation}
\mathcal{S}_{11}\hat{\Deltaup}_{\Gamma_1}\mathcal{S}_{11} = \frac{\mathcal{I}}{4} + \mathcal{K}_{\mathcal{S}_{11}\hat{\Deltaup}_{\Gamma_1}\mathcal{S}_{11}}\, ,
\end{equation}
where $\mathcal{K}_{\mathcal{S}_{11}\hat{\Deltaup}_{\Gamma_1}\mathcal{S}_{11}}$ is compact. Thus, given $f \in H^{1/2}(\Gamma_1)$ and the vector of coefficients of its linear expansion in patch basis function $\vec f$, the system of linear equations
\begin{equation}
\mat S_{11} \mat G_{1,\tilde{\lambda}\pi}^{-1} \hat{\tilde{\mat \Delta}}_1 \mat G_{1,\tilde{\lambda}\pi}^{-1}\mat S_{11} \vec x = \mat S_{11} \mat G_{1,\tilde{\lambda}\pi}^{-1} \hat{\tilde{\mat \Delta}}_1 \mat G_{1,\tilde{\lambda}\pi}^{-1} \vec f
\end{equation}
is non-singular and well-conditioned. Its unique solution coincides with the solution of the non-singular, but ill-conditioned, original problem $\mat S_{11} \vec x = \vec f$.

\subsection{A non-singular symmetric formulation}
As well-known from the literature \citep{kybic2005common}, when the conductivity in $\Omega_{N+1}$ (the exterior region) is zero, as is customary in EEG scenarios, the symmetric formulation is not well-posed due to a one-dimensional nullspace. This characteristic reflects the indeterminacy of the electrostatic potential, defined from the relation $\veg{E}=-\nabla V$ up to a constant \citep{kybic2005common}. Specifically, we find the one-dimensional, non-trivial nullspace of  matrix $\mat Z$ as the direction parallel to
\begin{equation}
\text{ker}(\mat Z) = \left[\bm{1}_{N_{V,1}}^\T \quad \bm{1}_{N_{V,2}}^\T
\quad ...
\quad \bm{1}_{N_{V,N}}^\T \quad \bm{0}_{N_{C,1}}^\T \quad \bm{0}_{N_{C,2}}^\T
\quad ...
\quad \bm{0}_{N_{C,N-1}}^\T
\right]^\T,
\end{equation}
following from
\begin{equation}
\begin{cases}
\mat N_{ij} \bm{1}_{N_{V,j}} &= \bm{0}_{N_{V,i}}\, ,\\
\mat D_{i-1,i} \bm{1}_{N_{V,i}} &= - \mat G_{\pi\pi,i-1} \bm{1}_{N_{C,i-1}} \, ,\\
\mat D_{ii} \bm{1}_{N_{V,i}} &= -\frac{1}{2}\,\mat G_{\pi\pi,i} \bm{1}_{N_{C,i}}\, , \\
\mat D_{i+1,i} \bm{1}_{N_{V,i}} &= \bm{0}_{N_{C,i+1}} \, ,
\end{cases}
\label{eqn:nullspace_system}
\end{equation}
which are the discrete counterparts of the eigenrelations
\begin{equation}
\begin{cases}
\mathcal{N}_{ij} \text{\textit{1}} &= 0\, \\
\mathcal{D}_{i-1,i} \text{\textit{1}} &= - \text{\textit{1}}\, , \\
\mathcal{D}_{ii} \text{\textit{1}} &= -\frac{1}{2} \text{\textit{1}}\, ,\\
\mathcal{D}_{i+1,i} \text{\textit{1}} &= 0\, .
\end{cases}
\label{eqn:nullspace_system_cont}
\end{equation}
The first relation of system \eqref{eqn:nullspace_system} identifies the non-trivial kernel of matrix $\mat N_{ij}$ discretizing the hypersingular operator applied to a function expanded in an interpolatory basis. 
The three last equations of \eqref{eqn:nullspace_system_cont} follow instead from a direct application of the representation theorem \citep[Theorem~3.1.1]{nedelec2001acoustic}. Their validity can be shown for example by applying the representation formula in \citep[Equation~3.1.7]{nedelec2001acoustic} to the scalar function $u(\veg r)$, defined as constant in the interior of $\Gamma_i$ and null elsewhere.

A deflation strategy is needed to obtain a non-singular problem (i.e., admitting unique solution), whose application is physically equivalent to fixing a reference for the evaluation of the potential. To the purpose of its implementation, similarly to what done in the previous section, we define the operator $\hat{\mathcal{N}}_{ii}:H^{1/2}(\Gamma_i)\rightarrow H^{-1/2}(\Gamma_i)$ \citep{steinbach1998construction} by the bilinear form
\begin{equation}
\big(\hat{\mathcal{N}}_{ii} v,w\big)_{L^2(\Gamma_i)} \coloneqq \big({\mathcal{N}}_{ii} v,w\big)_{L^2(\Gamma_i)} + \big(\text{\textit{1}},w\big)_{L^2(\Gamma_i)}\big(\text{\textit{1}},v\big)_{L^2(\Gamma_i)}
\end{equation}
for all $v,w\in H^{1/2}(\Gamma_i)$.
This modified hypersingular operator is bounded and $H^{1/2}-$elliptic \citep{steinbach2008numerical}. Moreover, the unique solution of
\begin{equation}
\hat{\mathcal{N}}_{ii} v = g
\end{equation}
is also a solution of
\begin{equation}
{\mathcal{N}}_{ii} v = g,
\end{equation}
provided that the solvability condition $\int_{\Gamma_i}g \dd S = 0$ is satisfied.
The discretization of $\hat{\mathcal{N}}_{ii}$ with pyramid functions as testing and expansion functions on $\Gamma_i$ gives rise to the invertible matrix
\begin{equation}
\hat{\mat N}_{ii} \coloneqq  {\mat N}_{ii} + \mat G_{i,\lambda\lambda}^\T \bm{1}_{N_{V,i}} \bm{1}_{N_{V,i}}^\T \mat G_{i,\lambda\lambda}.
\end{equation}
Among the infinite deflation choices available, one of them allows to retrieve the solution corresponding to a mean-value free potential on the exterior layer, a favorable choice in terms of compatibility with the most common measurement setups \citep{lei2017understanding} and based on theoretical justifications \citep{bertrand1985theoretical}. This is simply obtained by using $\hat{\mat N}_{NN}$ instead of $\mat N_{NN}$. Equivalently, by defining the vector
\begin{equation}
\vec \zeta \coloneqq  \left[\bm{0}_{N_{V,1}}^\T \quad \bm{0}_{N_{V,2}}^\T \quad ... \quad \mat G_{N,\lambda\lambda}^\T \bm{1}_{N_{V,N}} \quad \bm{0}_{N_{C,1}}^\T \quad ... \quad \bm{0}_{N_{C,N-1}}^\T \right]^\T,
\end{equation}
the unique solution of the linear system of equations
\begin{equation}
\hat{\mat Z}\begin{pmatrix}
\vec l \\ \vec p
\end{pmatrix} =
(\mat Z+\vec\zeta\vec\zeta^\T)\begin{pmatrix}
\vec l \\ \vec p
\end{pmatrix} = \begin{pmatrix}
\vec b \\ \vec c
\end{pmatrix}
\label{eqn:sys_def}
\end{equation}
is also a solution of
\begin{equation}
\mat Z\begin{pmatrix}
\vec l \\ \vec p
\end{pmatrix} = \begin{pmatrix}
\vec b \\ \vec c
\end{pmatrix}.
\label{eqn:sys_nondef}
\end{equation}
Moreover, $V_N$ is a mean-value free function, that is, we have
\begin{equation}
(\mat G_{N,\lambda\lambda}^\T \vec l_N)^\T \, \bm{1}_{N_{V,N}} = 0.
\label{eqn:def_cond}
\end{equation}
The above statements follow from the existence of a solution of equation \eqref{eqn:sys_nondef} satisfying condition \eqref{eqn:def_cond} and from the uniqueness of the solution of equation \eqref{eqn:sys_def}.
In the next section, we will define a preconditioning strategy for the well-posed formulation in \eqref{eqn:sys_def}.

\subsection{Our preconditioned, non-singular symmetric formulation}
\label{sec:newformulation}
The new formulation, resulting from the preconditioning of \eqref{eqn:sys_def}, for which we are going to prove the good conditioning properties, reads
\begin{equation}
\mat Z_\mr{p} \,\vec y\coloneqq
\mat M \, \mat Q\, \hat{\mat Z}\, \mat Q\, \mat G\, \mat P\, \mat G^\T\, \mat Q\, \hat{\mat Z}\, \mat Q\, \mat M\, \vec y = \mat M\, \mat Q\,  \hat{\mat Z}\,  \mat Q\, \mat G\,  \mat P\,  \mat G^\T\, \mat Q\, \begin{pmatrix}
\vec b \\ \vec c
\end{pmatrix}.
\label{eqn:fullformulation}
\end{equation}
The solution of the original problem \eqref{eqn:sys_def} is retrieved as
\begin{equation}
\begin{pmatrix}
\vec l \\ \vec p
\end{pmatrix} =  \mat Q\, \mat M\, \vec y.
\end{equation}
In the following, the matrices in \eqref{eqn:fullformulation} and a brief, intuitive explanation of their use in the formulation are introduced. The rigorous proof of the well conditioning of the formulation will be given in the next section.

The matrix $\mat Q$ is defined as
\begin{align}
\mat Q \coloneqq  \diag\bigg( &q_{V,1}\sqrt{R_1} \cdot \mat I_{N_{V,1},N_{V,1}}, q_{V,2}\sqrt{R_2} \cdot\mat I_{N_{V,2},N_{V,2}},...,q_{V,N}\sqrt{R_N} \cdot\mat I_{N_{V,N},N_{V,N}},\nonumber\\
&\frac{q_{C,1}}{\sqrt{R_1}} \cdot\mat I_{N_{C,1},N_{C,1}}, \frac{q_{C,2}}{\sqrt{R_2}} \cdot\mat I_{N_{C,2},N_{C,2}},...,\frac{q_{C,N-1}}{\sqrt{R_{N-1}}} \cdot\mat I_{N_{C,N-1},N_{C,N-1}}\bigg),
\end{align}
where $\mat I_{nm}\in\mathbb{R}^{n\times m}$ denotes the generalized (rectangular) identity matrix, i.e. $(\mat I_{nm})_{xy}=\delta_{xy}$, with $\delta_{xy}$ the Kronecker delta function.
The scalar coefficients $q_{V,i}$, $q_{C,i}$ are defined as \citep{ortizg.2018calderon}
\begin{align}
q_{V,i} &\coloneqq  \max(\sigma_i,\sigma_{i+1})^{-1/2}\\
q_{C,i} &\coloneqq  \min(\sigma_i,\sigma_{i+1})^{1/2}
\end{align}
and have been introduced in the formulation in order to cure the high-contrast breakdown identified in \Cref{sec:HCR_breakdown}. Indeed, as shown in \citep{ortizg.2018calderon}, the left and right multiplication of $\hat{\mat Z}$ by $\mat Q$ makes the symmetric formulation immune to this source of ill-conditioning.

If $\Gamma_{h,i}$ discretizes a spherical surface, then we define $R_i$ as the radius of the approximating sphere. Otherwise, $R_i$ represents half of the characteristic length of the inner volume delimited by $\Gamma_i$.

The matrix $\mat P$
\begin{equation}
\mat P \coloneqq  \diag\left( \frac{1}{R_1^2}\cdot \hat{{\mat\Delta}}_1^{-1}, \frac{1}{R_2^2}\cdot\hat{{\mat\Delta}}_2^{-1},...,\frac{1}{R_N^2}\cdot\hat{{\mat\Delta}}_N^{-1},
R_1^2\cdot\hat{\tilde{\mat\Delta}}_1, R_2^2\cdot\hat{\tilde{\mat\Delta}}_2,...,R_{N-1}^2\cdot\hat{\tilde{\mat\Delta}}_{N-1}\right)
\end{equation}
is the core of our preconditioning strategy.
The left multiplication of $\hat{\mat Z}$ by $\hat{\mat Z}$ and $\mat P$ provides the same preconditioning effect presented in \Cref{sec:POC} for the single-layer operator case, capable of overcoming the dense-discretization breakdown of the symmetric formulation.
The purpose of the introduction of the coefficients $R_i$ in the formulation is to make the spectra of the matrices $\mat Q\hat{\mat Z}\mat Q$ and $\mat P$ independent of the size of the geometry considered or, equivalently, independent of the unit of measure used for the definition of the mesh discretizing the head model.

Furthermore, we define the matrices
\begin{equation}
\mat G \coloneqq  \diag( \mat I_{N_{V,1},N_{V,1}}, \mat I_{N_{V,2},N_{V,2}},\dots, \mat I_{N_{V,N},N_{V,N}} \mat G_{1,\tilde{\lambda}\pi}^{-1}, \mat G_{2,\tilde{\lambda}\pi}^{-1},\dots,\mat G_{N-1,\tilde{\lambda}\pi}^{-1})
\end{equation}
and
\begin{equation}
\mat M \coloneqq  \diag( \mat G_{1,\lambda\lambda}^{-1/2}, \mat \mat G_{2,\lambda\lambda}^{-1/2},...,\mat G_{N,\lambda\lambda}^{-1/2}, \mat G_{1,\pi\pi}^{-1/2}, \mat G_{2,\pi\pi}^{-1/2},..., \mat \mat G_{N-1,\pi\pi}^{-1/2})\, .
\end{equation}
The left and right multiplication of the entire inner block $\mat Q \hat{\mat Z} \mat Q \mat G_L \mat P \mat G_R \mat Q \hat{\mat Z} \mat Q$ by $\mat M$ results in a matrix spectrally equivalent to the one discretizing the continuous formulation with a set of orthonormal basis functions.
One can be easily convinced of this by considering, for example, the simplifications leading to the equalities
\begin{equation}
\mat G_{1,\lambda\lambda}^{-1/2} \mat N_{11} \hat{{\mat\Delta}}_1^{-1} \mat N_{11} \mat G_{1,\lambda\lambda}^{-1/2} = \left(\mat G_{1,\lambda\lambda}^{-1/2} \mat N_{11} \mat G_{1,\lambda\lambda}^{-1/2}\right) \left(G_{1,\lambda\lambda}^{-1/2} \hat{{\mat\Delta}}_1 G_{1,\lambda\lambda}^{-1/2}\right)^{-1} \left(\mat G_{1,\lambda\lambda}^{-1/2} \mat N_{11} \mat G_{1,\lambda\lambda}^{-1/2}\right)
\end{equation}
and
\begin{multline}
\mat G_{1,\pi\pi}^{-1/2} \mat S_{11} \mat G_{1,\tilde{\lambda}\pi}^{-1} \hat{\tilde{\mat\Delta}}_1 \mat G_{1,\pi\tilde{\lambda}}^{-1} \mat S_{11} \mat G_{1,\pi\pi}^{-1/2} = \\ \left(\mat G_{1,\pi\pi}^{-1/2} \mat S_{11} \mat G_{1,\pi\pi}^{-1/2}\right) \left(\mat G_{1,\tilde{\lambda}\tilde{\lambda}}^{-1/2} \mat G_{1,\tilde{\lambda}\pi}G_{1,\pi\pi}^{-1/2}\right)^{-1}\left(\mat G_{1,\tilde{\lambda}\tilde{\lambda}}^{-1/2}\hat{\tilde{\mat\Delta}}_1 \mat G_{1,\tilde{\lambda}\tilde{\lambda}}^{-1/2}\right) \left(\mat G_{1,\tilde{\lambda}\tilde{\lambda}}^{-1/2} \mat G_{1,\tilde{\lambda}\pi}G_{1,\pi\pi}^{-1/2}\right)^{-1} \left(\mat G_{1,\pi\pi}^{-1/2} \mat S_{11} \mat G_{1,\pi\pi}^{-1/2}\right) 
\end{multline}
where the matrices
\begin{equation}
\left(\mat G_{1,\lambda\lambda}^{-1/2} \mat N_{11} \mat G_{1,\lambda\lambda}^{-1/2}\right),\quad \left(\mat G_{1,\pi\pi}^{-1/2} \mat S_{11} \mat G_{1,\pi\pi}^{-1/2}\right),\quad \left(\mat G_{1,\tilde{\lambda}\tilde{\lambda}}^{-1/2} \mat G_{1,\tilde{\lambda}\pi}G_{1,\pi\pi}^{-1/2}\right),\quad \left(G_{1,\lambda\lambda}^{-1/2} \hat{{\mat\Delta}}_1 G_{1,\lambda\lambda}^{-1/2}\right),\quad \text{and} \quad\left(\mat G_{1,\tilde{\lambda}\tilde{\lambda}}^{-1/2}\hat{\tilde{\mat\Delta}}_1 \mat G_{1,\tilde{\lambda}\tilde{\lambda}}^{-1/2}\right)
\nonumber
\end{equation}
are spectrally equivalent to the discretizations of the hypersingular operator, the single-layer operator, the identity and the modified Laplace-Beltrami operator in orthonormal bases.

It is worth noticing at this moment that, although the direct evaluation of matrix $\mat M$ could be expensive---indeed, it requires the square root decomposition of non-diagonal matrices, which can be expansive
---it can easily be avoided. For example, by using a simple similarity transformation, we have \citep{shewchuk1994introduction}
\begin{equation}
\text{eig}\left(\mat M \mat Q \hat{\mat Z} \mat Q \mat G_L \mat P \mat G_R \mat Q \hat{\mat Z} \mat Q \mat M\right) = \text{eig}\left(\mat M \mat M \mat Q \hat{\mat Z} \mat Q \mat G \mat P \mat G^\T \mat Q  \hat{\mat Z} \mat Q\right)\,,
\end{equation}
where the symbol $\text{eig}(\mat Z)$ denotes the set of eigenvalues of $\mat Z$.
When the preconditioned conjugate gradient method \citep{shewchuk1994introduction} is employed to solve the system, the similarity transformation will not change the convergence behavior, thus, preserving the favorable convergence properties of the original system in \eqref{eqn:fullformulation}.

\begin{proposition}[]
The coefficient matrix of the proposed formulation \eqref{eqn:fullformulation} is symmetric, positive-definite.
\end{proposition}
\begin{proof}
The matrix $\mat P$ can be written as $\mat P = \mat P_{\text{sqrt}}^\T \mat P_{\text{sqrt}}$ by virtue of its symmetric, positive-definiteness. Moreover, the matrices $\mat M$, $\mat Q$, $\mat Z$ are symmetric and invertible. Therefore,  we have the decomposition
\begin{equation}
\mat M \mat Q \hat{\mat Z} \mat Q \mat G \mat P  \mat G^\T \mat Q  \hat{\mat Z} \mat Q \mat M = \left( \mat P_{\text{sqrt}}  \mat G^\T \mat Q  \hat{\mat Z}  \mat Q \mat M \right)^\T \left( \mat P_{\text{sqrt}}  \mat G^\T \mat Q  \hat{\mat Z}  \mat Q \mat M \right)
\end{equation}
ensuring the symmetric, positive-definiteness of the proposed formulation.
\end{proof}

Before moving to the analysis of the conditioning properties of the new formulation, it is worth noting that the numerical scheme in \eqref{eqn:fullformulation} is refinement-free, that is, its implementation does not require the evaluation of integral operators on the dual mesh. Avoiding this operation sidesteps the computational burden of numerical integrations over the dual---and thus barycentrically refined mesh---leading to a significant advantage in terms of time required to build the formulation itself. In fact, in the proposed formulation \eqref{eqn:fullformulation} dual functions are only involved in the Gram matrices $(\mat G_{i,\tilde{\lambda}\pi})_{mn}=( \mat G_{i,\pi\tilde{\lambda}})_{nm} = \left( \tilde{\lambda}_{i,m}, \pi_{i,n} \right)_{L^2(\Gamma_i)}$ and in the Laplacian$(\tilde{\mat \Delta}_{i})_{mn} = \left( \nabla_{\Gamma_i}\tilde{\lambda}_{i,m}\nabla_{\Gamma_i}\tilde{\lambda}_{i,n}  \right)_{L^2(\Gamma_i)}$, for which we have, unlike the system matrices stemming from integral operators, analytic expressions allowing a rapid evaluation (see \citep{adrian2019refinementfree}) and \ref{sec:appendix_dualLap} for these expressions).

\subsection{Proof of well-conditioning}
\label{sec:well_cond}
In this section, we want to analyse the conditioning properties of the proposed formulation with respect to the two sources of instability identified in \cref{sec:conditionAnalysis}, that are dense-discretization and high-contrast.

For proving the high-contrast stability, we note that $\cond(\mat Q \hat{\mat Z} \mat Q) = \mathcal{O}(1)$ for $\mathit{CR}\coloneqq \max_i\left( \mathit{CR}_i \right)\rightarrow \infty$ directly follows from the discussion of the high-contrast spectral properties of $\mat Q \mat Z \mat Q$ in \citep{ortizg.2018calderon}. 
Then, we use the submultiplicativity of the condition number of matrix products
\begin{equation}
\text{cond}(\mat Z_\mr{p})\le\left(\text{cond}(\mat M)\right)^2\left(\text{cond}(\mat Q \hat{\mat Z} \mat Q)\right)^2\text{cond}(\mat G \mat P \mat G^\T)\, ,
\end{equation}
which follows from the submultiplicativity of the Euclidean norms in the definition of condition number employed, that is $\cond(\mat Z) = ||\mat Z||_2||\mat Z^{-1}||_2$ \citep[Equation~2.1.25]{quarteroni2008numerical}. Next, we study the limit $\lim_{\mathit{CR}\rightarrow \infty}$
\begin{equation}
\lim_{\mathit{CR}\rightarrow \infty}\text{cond}(\mat Z_\mr{p}) \le 
\lim_{\mathit{CR}\rightarrow \infty} 
\left(\text{cond}(\mat M)\right)^2
\left(\text{cond}(\mat Q \hat{\mat Z} \mat Q)\right)^2
\text{cond}(\mat G \mat P \mat G^\T).
\label{eqn:HClimit}
\end{equation}
Since the limits $\lim_{\mathit{CR}\rightarrow \infty}\text{cond}(\mat M)$, $\lim_{\mathit{CR}\rightarrow \infty}\text{cond}(\mat Q \hat{\mat Z} \mat Q)$, and $\lim_{\mathit{CR}\rightarrow \infty}\text{cond}(\mat G \mat P \mat G^\T)$ are finite, the resulting limit \eqref{eqn:HClimit}, given by the product of them \citep[Theorem~4.4]{rudin2008principles}, is finite, that is $\text{cond}(\mat Z_\mr{p}) = \mathcal{O}(1)$ as $\mathit{CR}\rightarrow \infty$, which concludes the proof of the high-contrast stability of our formulation.

For proving the dense-discretization stability of $\mat Z_\mr{p}$, we will leverage a spectral analysis of the matrix in \eqref{eqn:fullformulation}, first in the one-compartment case, then in the general $N-$compartment setup.
The study is held on a spherical multi-compartment model, but, as before (\Cref{sec:conditionAnalysis}), it can be extended to any geometry characterized by smooth boundaries without loss of generality.

\subsubsection{One-compartment case}
The matrix for which we are going to prove the dense-discretization stability reads
\begin{IEEEeqnarray}{rCl}
\mat Z_\mr{p} &=& 
\begin{pmatrix}
\alpha_1 \mat G_{1,\lambda\lambda}^{-1/2} {\hat{\mat N}}_{11}{\hat{\mat\Delta}}_1^{-1} {\hat{\mat N}}_{11}\mat G_{1,\lambda\lambda}^{-1/2} & 
\epsilon_1 \mat G_{1,\lambda\lambda}^{-1/2} {\hat{\Nm}}_{11}{\hat{\mat\Delta}}_1^{-1}{\Dxm}_{11} \mat G_{1,\pi\pi}^{-1/2} \nonumber \\
\epsilon_1 \mat G_{1,\pi\pi}^{-1/2}{\Dm}_{11}{\hat{\mat\Delta}}_1^{-1}{\hat{\Nm}}_{11}\mat G_{1,\lambda\lambda}^{-1/2} & \gamma_1\mat G_{1,\pi\pi}^{-1/2}{\Dm}_{11}{\hat{\mat\Delta}}_1^{-1}{\Dxm}_{11}\mat G_{1,\pi\pi}^{-1/2}
\end{pmatrix} \nonumber \\
&& +\> \begin{pmatrix}
\beta_1 \mat G_{1,\lambda\lambda}^{-1/2} {\Dxm}_{11}\mat G_{1,\tilde{\lambda}\pi}^{-1}{\hat{\tilde{\mat\Delta}}}_{1}\mat G_{1,\pi\tilde{\lambda}}^{-1}{\Dm}_{11} \mat G_{1,\lambda\lambda}^{-1/2} & 
\eta_1\mat G_{1,\lambda\lambda}^{-1/2}{\Dxm}_{11}\mat G_{1,\tilde{\lambda}\pi}^{-1}{\hat{\tilde{\mat\Delta}}}_{1}\mat G_{1,\pi\tilde{\lambda}}^{-1}{\Sm}_{11}\mat G_{1,\pi\pi}^{-1/2} \\
\eta_1 \mat G_{1,\pi\pi}^{-1/2}{\Sm}_{11}\mat G_{1,\tilde{\lambda}\pi}^{-1}{\hat{\tilde{\mat\Delta}}}_{1}\mat G_{1,\pi\tilde{\lambda}}^{-1}{\Dm}_{11}\mat G_{1,\lambda\lambda}^{-1/2} & 
\delta_1\mat G_{1,\pi\pi}^{-1/2}{\Sm}_{11}\mat G_{1,\tilde{\lambda}\pi}^{-1}{\hat{\tilde{\mat\Delta}}}_{1}\mat G_{1,\pi\tilde{\lambda}}^{-1}{\Sm}_{11}\mat G_{1,\pi\pi}^{-1/2}
\end{pmatrix}
\label{eqn:proof1_completeMatrix}
\end{IEEEeqnarray}
with the scalings
\begin{align}
\alpha_1 &\coloneqq  q_{V,1}^4(\sigma_1+\sigma_2)^2\,, \quad\quad\quad\quad\quad\quad\quad\quad\quad\,\,\, \,\,
\beta_1 \coloneqq  4R_1^2(q_{V,1}q_{C,1})^2\,,\nonumber\\
\epsilon_1 &\coloneqq  -\frac{2}{R_1}q_{V,1}^2(q_{V,1}q_{C,1})(\sigma_1+\sigma_2)\,,\quad\quad\quad\quad
\delta_1 \coloneqq  q_{C,1}^4(\sigma_1^{-1}+\sigma_2^{-1})^2\,,  
\nonumber\\
\gamma_1 &\coloneqq  \frac{4}{R_1^2}(q_{V,1}q_{C,1})^2\,, \quad\quad\quad\quad \quad\quad\quad\quad\quad\quad
\eta_1 \coloneqq  -2R_1q_{C,1}^2(q_{V,1}q_{C,1})(\sigma_1^{-1}+\sigma_2^{-1})\,.\nonumber.
\end{align}

For the reasons presented in \Cref{sec:POC}, the matrix $\mat G_{1,\pi\pi}^{-1/2}{\Sm}_{11}{\hat{\mat\Delta}}_{\pi,1}{\Sm}_{11}\mat G_{1,\pi\pi}^{-1/2}$, discretizing the second kind operator $\mathcal{S}_{11}\hat{\Deltaup}_{\Gamma_1}\mathcal{S}_{11} = \frac{\mathcal{I}}{4} + \mathcal{K}_{\mathcal{S}_{11}\hat{\Deltaup}_{\Gamma_1}\mathcal{S}_{11}}$, is well-conditioned and its spectrum accumulates at $1/4$.
An accurate study is needed to understand the nature of all the other blocks.

\begin{proposition}[]
The matrix $\mat G_{1,\lambda\lambda}^{-1/2} {\hat{\mat N}}_{11}{\hat{\mat\Delta}}_1^{-1} {\hat{\mat N}}_{11}\mat G_{1,\lambda\lambda}^{-1/2}$ is well-conditioned and its spectrum accumulates at $1/4$.
\end{proposition}
\begin{proof}
From the Calder\'{o}n relations, by expanding the products $\mathcal{S}\hat{\mathcal{N}}$ and $\hat{\mathcal{N}}\mathcal{S}$, we obtain that
\begin{align}
-\mathcal{S}\hat{\mathcal{N}} &= \frac{\mathcal{I}}{4} + \mathcal{K}_{\mathcal{S}\hat{\mathcal{N}}}\\
-\hat{\mathcal{N}}\mathcal{S} &= \frac{\mathcal{I}}{4} + \mathcal{K}_{\hat{\mathcal{N}}\mathcal{S}},
\end{align}
where $\mathcal{K}_{\mathcal{S}\hat{\mathcal{N}}}$, $\mathcal{K}_{\hat{\mathcal{N}}\mathcal{S}}$ are compact operators. By exploiting these relations for simplifying the product $\mathcal{S}\hat{\Deltaup}_{\Gamma}\mathcal{S}\hat{\mathcal{N}}\hat{\Deltaup}_{\Gamma}^{-1}\hat{\mathcal{N}}$, it is found that $\hat{\mathcal{N}}\hat{\Deltaup}_{\Gamma}^{-1}\hat{\mathcal{N}}$ is a second kind operator with accumulation point at $1/4$. Indeed
\begin{align}
\mathcal{S}\hat{\Deltaup}_{\Gamma}\mathcal{S}\hat{\mathcal{N}}\hat{\Deltaup}_{\Gamma}^{-1}\hat{\mathcal{N}} &= -\mathcal{S}\hat{\Deltaup}_{\Gamma} \left(\frac{\mathcal{I}}{4} + \mathcal{K}_{\mathcal{S}\hat{\mathcal{N}}} \right)\hat{\Deltaup}_{\Gamma}^{-1}\hat{\mathcal{N}} \nonumber\\ &=-\frac{1}{4}\mathcal{S}\hat{\mathcal{N}}-\mathcal{S}\hat{\Deltaup}_{\Gamma}\mathcal{K}_{\mathcal{S}\hat{\mathcal{N}}}\hat{\Deltaup}_{\Gamma}^{-1}\hat{\mathcal{N}}
\nonumber\\ &= \phantom{-}\frac{\mathcal{I}}{16} + \frac{1}{4}\mathcal{K}_{\mathcal{S}\hat{\mathcal{N}}}-\mathcal{S}\hat{\Deltaup}_{\Gamma}\mathcal{K}_{\mathcal{S}\hat{\mathcal{N}}}\hat{\Deltaup}_{\Gamma}^{-1}\hat{\mathcal{N}},
\label{eqn:NLN_proof}
\end{align}
where $\mathcal{S}\hat{\Deltaup}_{\Gamma}\mathcal{K}_{\mathcal{S}\hat{\mathcal{N}}}\hat{\Deltaup}_{\Gamma}^{-1}\hat{\mathcal{N}}$ is compact as the product of compact and of bounded linear operators is compact \citep[Theorem~1.5]{colton2013integral}. Therefore, equation \eqref{eqn:NLN_proof} implies $\hat{\mathcal{N}}\hat{\Deltaup}_{\Gamma}^{-1}\hat{\mathcal{N}} = \mathcal{I}/4 + \mathcal{K}_{\hat{\mathcal{N}}\hat{\Deltaup}_{\Gamma}^{-1}\hat{\mathcal{N}}}$. The discretization of this second-kind operator with an orthonormal set of basis functions, as in $\mat G_{1,\lambda\lambda}^{-1/2} {\hat{\mat N}}_{11}{\hat{\mat\Delta}}_1^{-1} {\hat{\mat N}}_{11}\mat G_{1,\lambda\lambda}^{-1/2}$, results thus in a well-conditioned matrix with its eigenvalues accumulating at $1/4$. 
\end{proof}

From a spherical harmonic analysis held on a sphere of radius $R_1$, we recognize that the eigenvalues of the operator $\mathcal{D}^*\Deltaup_{\Gamma}\mathcal{D}$ are the product of the eigenvalues of $\mathcal{D}^*$, $\Deltaup_{\Gamma}$, and $\mathcal{D}$,
\begin{equation}
\frac{1}{R_1^2}\frac{1}{2(2n+1)}\,n(n+1)\,\frac{1}{2(2n+1)} = \frac{1}{R_1^2} \frac{n^2+n}{16 \,n^2+16 \,n +4}\,,
\label{eqn:DsLtrum}
\end{equation}
since the operators $\mathcal{D}^*$, $\Deltaup_{\Gamma}$, and $\mathcal{D}$ have the same eigenvectors \citep{hsiao1994error, darbas2006generalized}.
Therefore, by analyzing the limit of \eqref{eqn:DsLtrum} for $n\rightarrow \infty$, it is possible to state that $\mathcal{D}^*\Deltaup_{\Gamma}\mathcal{D}$ is a second-kind operator with accumulation point $1/(16\,R_1^2)$. Hence, since the modified operator $\hat{\Deltaup}_{\Gamma}$ provides similar spectral properties to $\Deltaup_{\Gamma}$, the discretized form $R_1^2\,\mat G_{1,\lambda\lambda}^{-1/2} {\Dxm}_{11}\mat G_{1,\tilde{\lambda}\pi}^{-1}{\hat{\tilde{\mat\Delta}}}_{1}\mat G_{1,\pi\tilde{\lambda}}^{-1}{\Dm}_{11} \mat G_{1,\lambda\lambda}^{-1/2}$ on spherical geometries is well-conditioned and its spectrum accumulates at $1/16$.

Similarly, the expression of the eigenvalues of the two operators $\mathcal{D}^*\Deltaup_{\Gamma}\mathcal{S}$ and $\mathcal{S}\Deltaup_{\Gamma}\mathcal{D}$ is given by
\begin{equation}
-\frac{1}{R_1}\frac{1}{2(2n+1)}\,n(n+1)\,\frac{1}{2n+1} = -\frac{1}{R_1} \frac{n^2+n}{8 \,n^2+8 \,n +2}\,,
\label{eqn:DsLStrum}
\end{equation}
from which we deduce that $\mathcal{D}^*\Deltaup_{\Gamma}\mathcal{S}$ and $\mathcal{S}\Deltaup_{\Gamma}\mathcal{D}$ are second-kind operators with eigenvalues accumulating at $-1/(8\,R_1)$ and the matrices $R_1\,\mat G_{1,\lambda\lambda}^{-1/2}{\Dxm}_{11}\mat G_{1,\tilde{\lambda}\pi}^{-1}{\hat{\tilde{\mat\Delta}}}_{1}\mat G_{1,\pi\tilde{\lambda}}^{-1}{\Sm}_{11}\mat G_{1,\pi\pi}^{-1/2}$ and $R_1\,\mat G_{1,\pi\pi}^{-1/2}{\Sm}_{11}\mat G_{1,\tilde{\lambda}\pi}^{-1}{\hat{\tilde{\mat\Delta}}}_{1}\mat G_{1,\pi\tilde{\lambda}}^{-1}{\Dm}_{11}\mat G_{1,\lambda\lambda}^{-1/2}$ discretizing the operators on a sphere are well-conditioned with spectra accumulating at $1/8$.
Since a smooth variation from a spherical geometry only results in a compact perturbation \citep{kolm2003quadruple, sloan1992error}, the results mentioned above hold true also for non-spherical, smooth geometries, such as the ones considered in this work.

\begin{proposition}[]\label{prop_compactnessDLD}
The matrices $\mat G_{1,\lambda\lambda}^{-1/2} {\hat{\Nm}}_{11}{\hat{\mat\Delta}}_1^{-1}{\Dxm}_{11} \mat G_{1,\pi\pi}^{-1/2}$, $\mat G_{1,\pi\pi}^{-1/2}{\Dm}_{11}{\hat{\mat\Delta}}_1^{-1}{\hat{\Nm}}_{11}\mat G_{1,\lambda\lambda}^{-1/2}$, and $\mat G_{1,\pi\pi}^{-1/2}{\Dm}_{11}{\hat{\mat\Delta}}_1^{-1}{\Dxm}_{11}\mat G_{1,\pi\pi}^{-1/2}$ discretize compact operators, that is, their spectra accumulate at $0$.
\end{proposition}
\begin{proof}
The operators $\hat{\mathcal{N}}\hat{\Deltaup}_{\Gamma}^{-1}\mathcal{D}^*$, $\mathcal{D}\hat{\Deltaup}_{\Gamma}^{-1}\hat{\mathcal{N}}$, and $\mathcal{D}\hat{\Deltaup}_{\Gamma}^{-1}\mathcal{D}^*$ are compact, as they can be written as the product of second kind and compact operators \citep[Theorem~1.5]{colton2013integral}, as
\begin{align}
\hat{\mathcal{N}}\hat{\Deltaup}_{\Gamma}^{-1}\mathcal{D}^* &= \hat{\mathcal{N}}\hat{\Deltaup}_{\Gamma}^{-1}\hat{\mathcal{N}}\hat{\mathcal{N}}^{-1}\mathcal{D}^*,\nonumber\\
\mathcal{D}\hat{\Deltaup}_{\Gamma}^{-1}\hat{\mathcal{N}} &= \mathcal{D}\hat{\mathcal{N}}^{-1}\hat{\mathcal{N}}\hat{\Deltaup}_{\Gamma}^{-1}\hat{\mathcal{N}},\nonumber\\
\mathcal{D}\hat{\Deltaup}_{\Gamma}^{-1}\mathcal{D}^* &= \mathcal{D}\hat{\mathcal{N}}^{-1}\hat{\mathcal{N}}\hat{\Deltaup}_{\Gamma}^{-1}\hat{\mathcal{N}} \hat{\mathcal{N}}^{-1}\mathcal{D}^*,\nonumber
\end{align}
where $\hat{\mathcal{N}}^{-1}\mathcal{D}^*$ and $\mathcal{D}\hat{\mathcal{N}}^{-1}$ are compact operators, composed out of bounded operators, where at least one is compact \citep[Theorem~1.5]{colton2013integral}.
Therefore, their discretization results in ill-conditioned matrices with vanishing spectra.
\end{proof}

Given the considerations above, matrix \eqref{eqn:proof1_completeMatrix} can be written as
\begin{equation}
\begin{pmatrix}
(\alpha_1 + \beta_1/4)/4\,\mat I_{N_{V,1},N_{V,1}} & 
\eta_1/8\,\mat I_{N_{V,1},N_{C,1}} \\
\eta_1/8\,\mat I_{N_{C,1},N_{V,1}} & \delta_1/4\,\mat I_{N_{C,1},N_{C,1}}
\end{pmatrix} + \begin{pmatrix}
\mat K_{N_{V,1},N_{V,1}} & \mat K_{N_{V,1},N_{C,1}} \\
\mat K_{N_{C,1},N_{V,1}} & \mat K_{N_{C,1},N_{C,1}}
\end{pmatrix},
\label{eqn:proof_1}
\end{equation}
where $\mat K_{m,n} \in \mathbb{R}^{m\times n}$ represents the discretization of a compact operator.
As is clear from their expression, $\alpha_1$ and $\beta_1$ are positive scalar coefficients, so that the principal part of the top-left block of \eqref{eqn:proof1_completeMatrix} cannot be canceled.
The second term in the summation \eqref{eqn:proof_1} is the discretization of a block operator with compact blocks, and thus is compact (proof in \ref{sec:proof_blockcompact}), with singular values accumulating at zero when increasing the number of degrees of freedom of the system, and as a consequence it does not influence the spectral properties of the system asymptotically. 
Therefore, in order to analyze the boundness of the condition number of the system away from singularities, it is sufficient to study the spectral behaviour of the principal part in \eqref{eqn:proof_1}.
This can be established by the analytical evaluation of the eigenvalues of the principal term in \eqref{eqn:proof_1}, performed by means of a Schur analysis \citep{strang2009introduction}.
In particular, in the case $N_{C,1}\ge N_{V,1}$ it is found that
\begin{align}
\uplambda_1 &= \frac{\delta_1}{4}, \quad &\text{with multiplicity $N_{C,1}- N_{V,1}$}\\
\uplambda_2 &= \frac{1}{32}\left( 4\alpha_1+\beta_1+4\delta_1-\sqrt{(4\alpha_1+\beta_1-4\delta_1)^2+16\eta_1^2
}\right), \quad &\text{with multiplicity $N_{V,1}$}\\
\uplambda_3 &= \frac{1}{32}\left( 4\alpha_1+\beta_1+4\delta_1+\sqrt{(4\alpha_1+\beta_1-4\delta_1)^2+16\eta_1^2
}\right), \quad &\text{with multiplicity $N_{V,1}$};
\end{align}
in the case $N_{C,1}< N_{V,1}$ instead, the eigenvalues are
\begin{align}
\uplambda_1 &= \frac{4\alpha_1+\beta_1}{16}, \quad &\text{with multiplicity $N_{V,1}- N_{C,1}$}\\
\uplambda_2 &= \frac{1}{32}\left( 4\alpha_1+\beta_1+4\delta_1-\sqrt{(4\alpha_1+\beta_1-4\delta_1)^2+16\eta_1^2
}\right), \quad &\text{with multiplicity $N_{C,1}$}\\
\uplambda_3 &= \frac{1}{32}\left( 4\alpha_1+\beta_1+4\delta_1+\sqrt{(4\alpha_1+\beta_1-4\delta_1)^2+16\eta_1^2
}\right), \quad &\text{with multiplicity $N_{C,1}$}.
\end{align}
Hence, the condition number of the principal part of $\mat Z$, that is symmetric, positive-definite, is given by 
\begin{equation}
\max(\uplambda_1,\uplambda_2,\uplambda_3) / \min(\uplambda_1,\uplambda_2,\uplambda_3),\nonumber
\end{equation}
independent of mesh refinement. Therefore, the condition number of the overall matrix $\mat Z$ when asymptotically increasing the number of degrees of freedom of the system is bounded and the proposed system is well-conditioned with respect to dense discretizations.

Moreover, given the asymptotic behaviours
\begin{align}
&\alpha_1 =  \mathcal{O}(1)\,, \quad\quad\quad\quad\quad\quad\quad\,\,\, \,\,
\beta_1 =  \mathcal{O}(\mathit{CR}_{1}^{-1})\,,\nonumber\\
&\abs{\epsilon_1} =  \mathcal{O}(\mathit{CR}_{1}^{-1/2})\,,\quad\quad\quad\quad\quad
\delta_1 =  \mathcal{O}(1)\,,  
\nonumber\\
&\gamma_1 =  \mathcal{O}(\mathit{CR}_{1}^{-1})\,, \quad \quad\quad\quad\quad\quad\,\,
\abs{\eta_1} =  \mathcal{O}(\mathit{CR}_{1}^{-1/2})\,.\nonumber
\end{align}
valid in the limit $\mathit{CR}_{1}\rightarrow \infty$, we observe that the condition number of the principal part of $\mat Z$ tends to unity in the high-contrast regime, as $\alpha_1$ approaches $\delta_1$, that is $\left(\max(\uplambda_1,\uplambda_2,\uplambda_3) / \min(\uplambda_1,\uplambda_2,\uplambda_3)\right)- 1 = \mathcal{O}(\mathit{CR}_{1}^{-1/2})$ for $\mathit{CR}_{1}\rightarrow \infty$.

\subsubsection{$N-$compartment case}
For the generic, $N-$layered geometry, the matrix $\mat Z_\mr{p}$ in \eqref{eqn:fullformulation} can be written as
\begin{equation}
\begin{pmatrix}
(\alpha_1 + \beta_1/4)/4\,\mat I_{N_{V,1},N_{V,1}} & & & \eta_1/8\,\mat I_{N_{V,1},N_{C,1}} & & \\
& \ddots &&& \ddots &\\
&& (\alpha_N + \beta_N/4)/4\,\mat I_{N_{V,N},N_{V,N}}&&& \eta_N/8\,\mat I_{N_{V,N},N_{C,N}}\\
\eta_1/8\,\mat I_{N_{C,1},N_{V,1}} &&& \delta_1/4\,\mat I_{N_{C,1},N_{C,1}}&&\\
& \ddots &&& \ddots &\\
&& \eta_N/8\,\mat I_{N_{C,N},N_{V,N}}&&& \delta_N/4\,\mat I_{N_{C,N},N_{C,N}}\\
\end{pmatrix} + \mat K
\label{eqn:proof_N}
\end{equation}
where the scaling coefficients are
\begin{align}
\alpha_n &\coloneqq  q_{V,n}^4(\sigma_n+\sigma_{n+1})^2\,,\quad\quad\quad\quad\,\,\, \delta_n \coloneqq  q_{C,n}^4(\sigma_n^{-1}+\sigma_{n+1}^{-1})^2\,,\nonumber\\
\beta_n &\coloneqq  4R_n^2(q_{V,n}q_{C,n})^2\,,\quad\quad\quad\quad\quad \eta_n \coloneqq  -2R_nq_{C,n}^2(q_{V,n}q_{C,n})(\sigma_n^{-1}+\sigma_n^{-1})\,,\nonumber.
\end{align}
with subscript $n=1,\dots,N$.
The matrix $\mat K$ is a $(2N\times 2N)-$block matrix, discretizing a compact, block operator (see \ref{sec:proof_blockcompact}). Indeed, each block of $\mat K$ is a linear combination of matrices discretizing compact operators, either in the form,
\begin{equation}
\hat{\mathcal{N}}_{ij}\hat{\Deltaup}^{-1}_{\Gamma_j}\mathcal{D}^*_{jk},\quad\mathcal{D}_{ij}\hat{\Deltaup}^{-1}_{\Gamma_j}\hat{\mathcal{N}}_{jk},\quad\mathcal{D}_{ij}\hat{\Deltaup}^{-1}_{\Gamma_j}\mathcal{D}^*_{jk},
\end{equation}
whose compactness has already been discussed in the proof of  \Cref{prop_compactnessDLD},
or in the form
\begin{equation}
\hat{\mathcal{N}}_{ij}\hat{\Deltaup}^{-1}_{\Gamma_j}\hat{\mathcal{N}}_{jk},\quad {\mathcal{S}}_{ij}\hat{\tilde{\Deltaup}}_{\Gamma_j}{\mathcal{S}}_{jk},\quad {\mathcal{D}}^*_{ij}\hat{\tilde{\Deltaup}}_{\Gamma_j}{\mathcal{D}}_{jk},\quad {\mathcal{D}}^*_{ij}\hat{\tilde{\Deltaup}}_{\Gamma_j}{\mathcal{S}}_{jk},\quad {\mathcal{S}}_{ij}\hat{\tilde{\Deltaup}}_{\Gamma_j}{\mathcal{D}}_{jk},
\label{eqn:compact_offdiag_operators}
\end{equation}
with $i\ne j$ or $j\ne k$, resulting from the product of bounded and compact operators \citep[Theorem~1.5]{colton2013integral}. 

As in the one-compartment case, the analytic spectral analysis of the principal part of \eqref{eqn:proof_N}, when fixing the number of compartments $N$, is useful to determine the asymptotic spectral behaviour of the overall matrix $\mat Z_\mr{p}$ when increasing the number of unknowns.
In the general case, we can argue that, since the coefficients $\alpha_n$, $\beta_n$, $\delta_n$, and $\eta_n$ are independent from the refinement parameter $h$, also the condition number of the principal part of \eqref{eqn:proof_N} is $h-$independent.

In addition, we can perform an asymptotic analysis of the principal part of \eqref{eqn:proof_N} for the contrast ratio going to infinity, providing some insights in the spectral behaviour of the proposed formulation in the high-contrast and dense-discretization limit. We can assume, for example, an $N-$compartment structure (with $N$ odd) with conductivities $\sigma_{1+2n} = \sigma_\text{high}$ and $\sigma_{2+2n} = \sigma_\text{low}$, for $n$ running from $0$ to $(N-1)/2$, satisfying $\mathit{CR}=\sigma_\text{high} / \sigma_\text{low} \rightarrow \infty$. Then, the behaviour of the scalar scalings is
\begin{align}
&\alpha_i =  \mathcal{O}(1)\,, \quad\quad\quad\quad\quad\quad\quad\,\,\, \,\,
\beta_i =  \mathcal{O}(\mathit{CR}^{-1})\,,\nonumber\\
&\abs{\eta_i} =  \mathcal{O}(\mathit{CR}^{-1/2})\,,\quad\quad\quad\quad\quad
\delta_i =  \mathcal{O}(1)\,,
\nonumber
\end{align}
resulting in the behaviour of the condition number of the principal part of $\mat Z_\mr{p}$ as $\left(\text{cond}(\mat Z_\mr{p}-\mat K)-1\right)=\mathcal{O}(\mathit{CR}^{-1/2})$ as $\mathit{CR}\rightarrow\infty$.

The asymptotic form above highlights the fact that the designed preconditioning is optimized for high-contrast structures. Nevertheless, the proposed formulation remains stable in the whole range of conductivity ratios commonly employed to model the head biological tissues, as shown in the numerical results \Cref{sec:result}.

\section{Numerical results} \label{sec:result}
The numerical results reported in this section showcase the efficacy of the proposed formulation and offer a comparison with the standard symmetric formulation. To this end, we first apply the numerical schemes to the canonical, spherical, head model, for which the analytic solution is available as a benchmark \citep{demunck1988potential,zhang1995fast}. Then, we move on to a more realistic head model, extracted from magnetic resonance imaging data, over which we solve the inverse EEG problem.

\begin{figure}
\subfloat[\label{subfig-1:nested}]{%
  \includegraphics[width=0.3
\columnwidth]{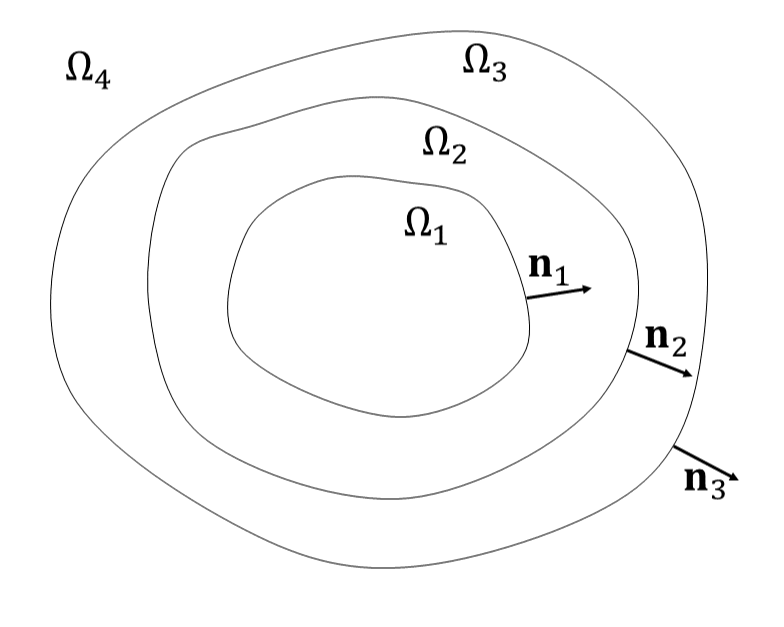}
}
\subfloat[\label{subfig-2:sphere_model}]{%
  \includegraphics[width=0.4
\columnwidth]{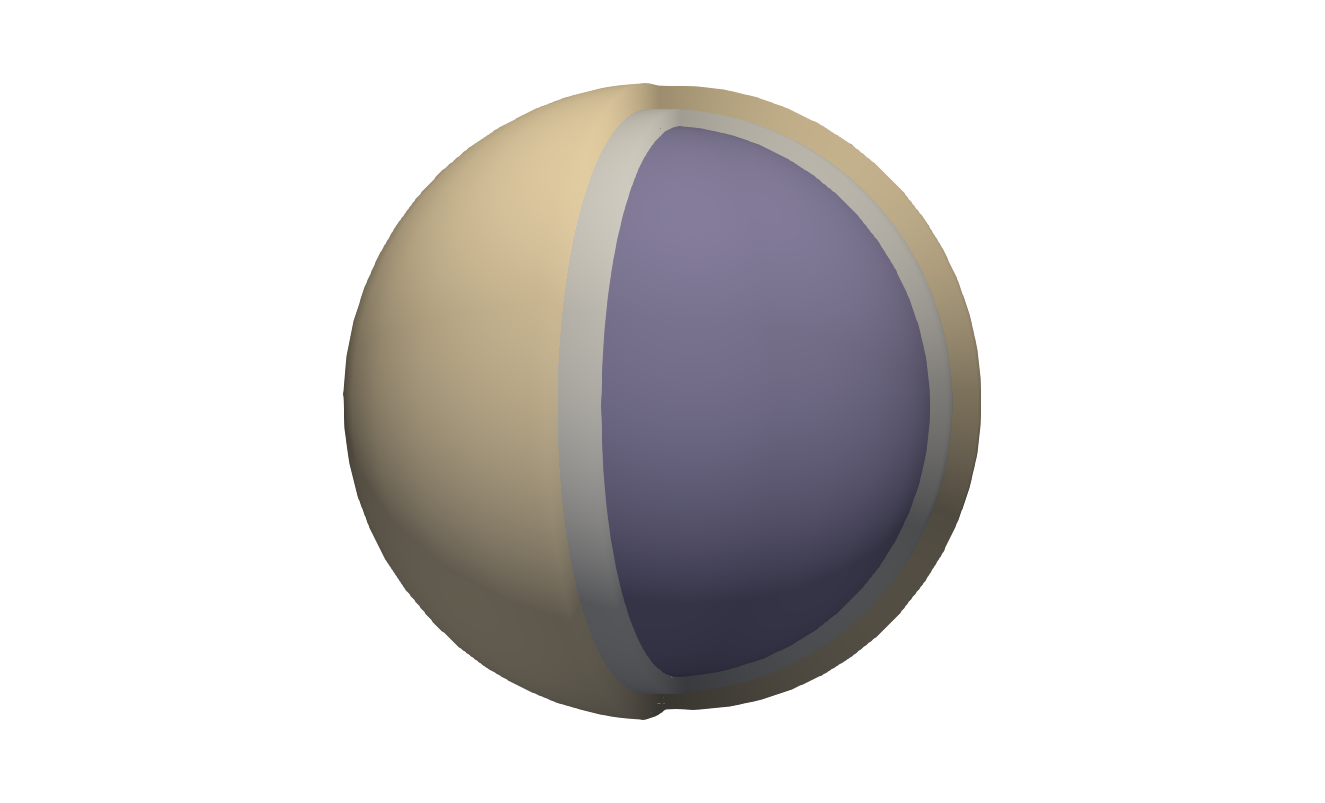}
}
\hfill
\subfloat[\label{subfig-3:mri_model}]{%
  \includegraphics[width=0.4\columnwidth]{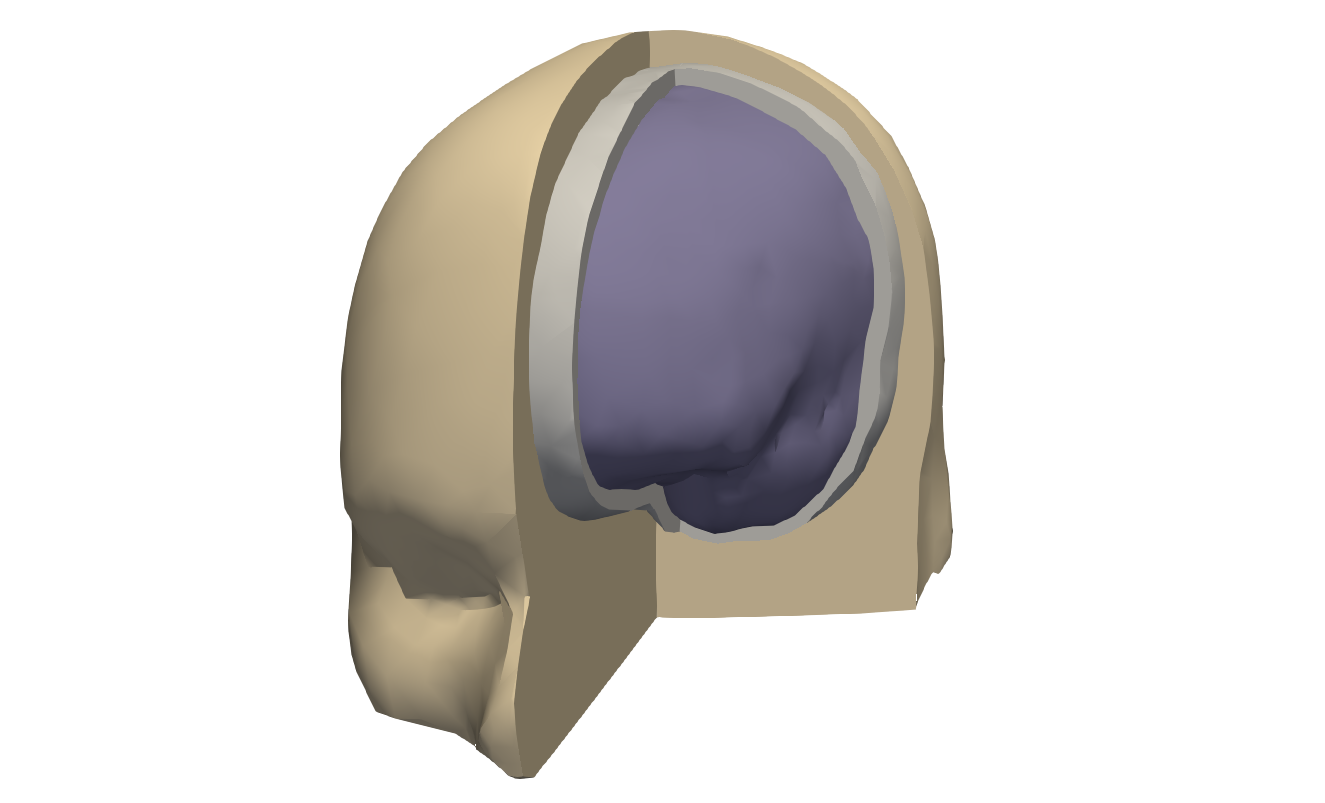}
}
\caption{(a) Schematic representation of the geometry under study. (b) The spherical and (c) the realistic, MRI-obtained head models. The three-compartments represent the brain, the skull and the skin.}
\end{figure}

\subsection{The spherical head model} \label{sec:results_sphere}
A three-compartment head model is considered in this section, delimited by the spherical surfaces $\Gamma_1$, $\Gamma_2$, and $\Gamma_3$ with radii $R_1 = $ \SI{8.7}{cm}, $R_2 = $ \SI{9.2}{cm}, and $R_3 = $ \SI{10}{cm}.
In this first set of experiments, the source employed is a single dipole, placed inside $\Omega_1$ at a distance of \SI{4.2}{cm} from $\Gamma_1$ and radially directed.
The conductivities of the three layers, modelling the brain, the skull and the scalp, have been set to $\sigma_1=$ \SI[parse-numbers = false]{\frac{1}{3}}{S/m},  $\sigma_2=$ \SI[parse-numbers = false]{\frac{1}{240}}{S/m}, and  $\sigma_3=$ \SI[parse-numbers = false]{\frac{1}{3}}{S/m}, as typical in these models.

As first assessment, we verify the convergence rate of our new formulation, similar as in the unpreconditioned symmetric formulation case with same accuracy levels as shown in \Cref{fig:sphere_accuracy}.

Next we assess the efficacy of our preconditioner.
In \Cref{fig:sphere_cn}, the variation of the condition number as a function of the inverse refinement parameter $1/h$ is reported for the two formulations considered. The efficacy of the preconditioning strategy is demonstrated by the constant conditioning in refinement, while the condition number of the symmetric formulation grows with the discretization as $\mathcal{O}(h^{-2})$.
A similar behaviour is reflected in \Cref{subfig-1:sphere_NIt}, where the number of iterations to solve the system up to a fixed level of accuracy is reported for the two formulations. The iterative method employed to solve the system in the two cases is the conjugate gradient-squared (CGS) solver \citep{sonneveld1989cgs}.

The linear system arising from the preconditioning scheme presented in this work can also be solved by means of a preconditioned conjugate gradient (PCG) scheme \citep{hestenes1952methods,steihaug1982conjugate}, by virtue of its symmetric, positive, definiteness properties.
The preconditioning matrix is $\mat M^2$, as explained in \Cref{sec:newformulation}.
The number of matrix-vector products required for the solution, corresponding to the number of iterations, is shown in \Cref{subfig-2:sphere_NMVP} in green. This can be compared with the number of matrix-vector products (two per iteration) required to solve the symmetric formulation and the proposed preconditioned formulation by means of the conjugate gradient squared solver, at the same level of accuracy, shown in the same figure. 

The efficacy of the proposed preconditioning has to be tested also for different values of conductivity contrasts between adjacent compartments. Experimental evidences from in vivo measurements have shown that the conductivity ratio between brain, skull, and scalp range between $(1{:}1/15{:}1)$ and $(1{:}1/80{:}1)$ \citep{gonalves2003vivo, zhang2006estimation, clerc2005vivo}. So, we evaluated the stability of our formulation for conductivity ratios spanning from $(1{:}1/10{:}1)$ to $(1{:}1/100{:}1)$. \Cref{subfig-1:sphere_NIt_CR}, showing the number of iterations of the CGS solver as a function of the conductivity contrast ratio, gives evidence of the preconditioning effect obtained. The number of matrix-vector products required to solve the proposed formulation by means of the conjugate gradient scheme is also shown in \Cref{subfig-2:sphere_NMVP_CR} and compared with the one required for solving both the symmetric formulation and the proposed scheme by means of the CGS solver.

\begin{figure}
\centerline{\includegraphics[width=0.5\columnwidth]{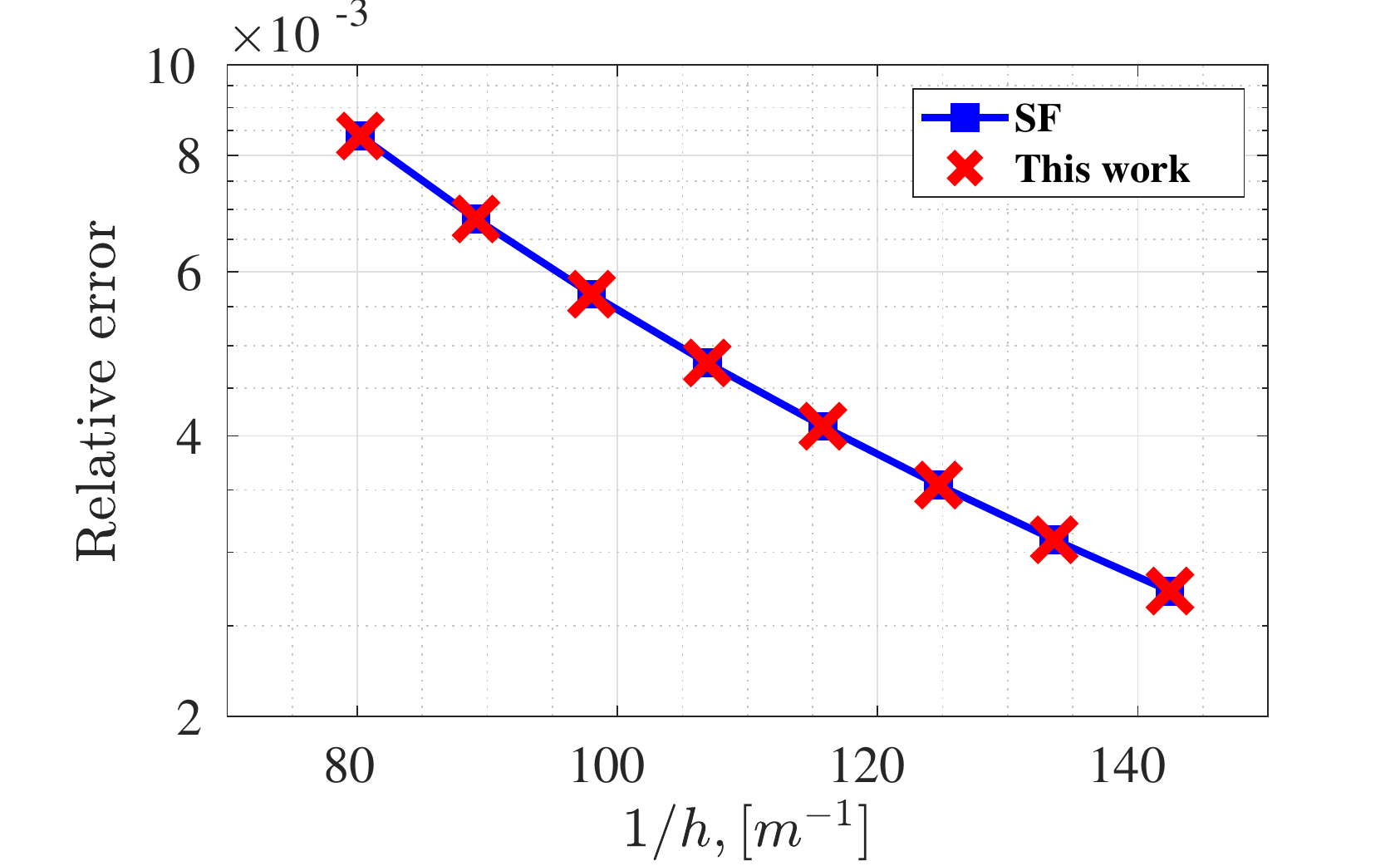}}
\caption{Relative error between the numerical solution and the analytic solution for the potential on $\Gamma_3$ against the inverse mesh refinement parameter $1/h$: comparison between the symmetric formulation (SF) and this work.}
\label{fig:sphere_accuracy}
\end{figure}

\begin{figure}
\centerline{\includegraphics[width=0.5\columnwidth]{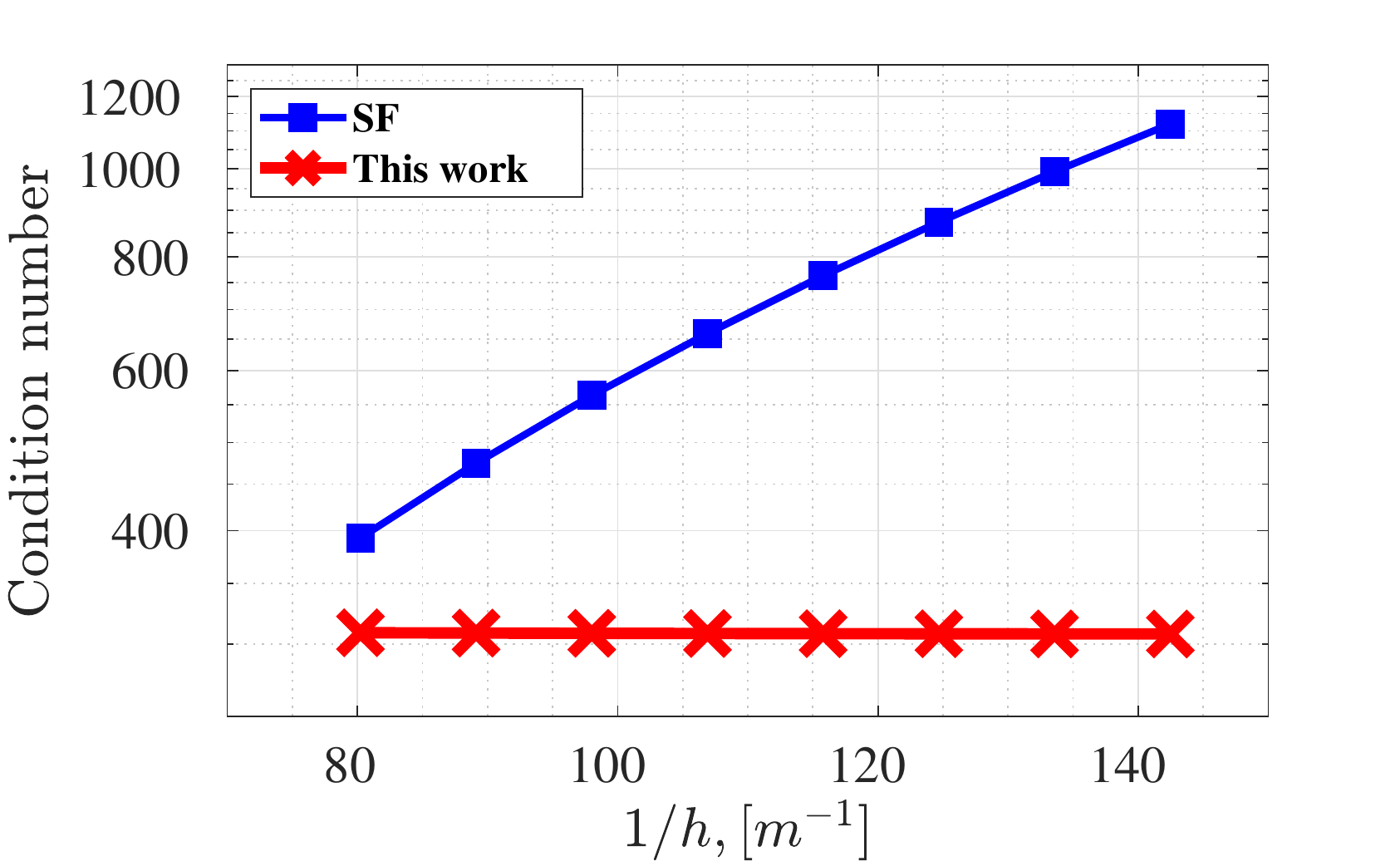}}
\caption{Condition number as a function of the inverse mesh refinement parameter $1/h$: comparison between the symmetric formulation (SF) and this work.}
\label{fig:sphere_cn}
\end{figure}

\begin{figure}
\subfloat[\label{subfig-1:sphere_NIt}]{%
  \includegraphics[width=0.5
\columnwidth]{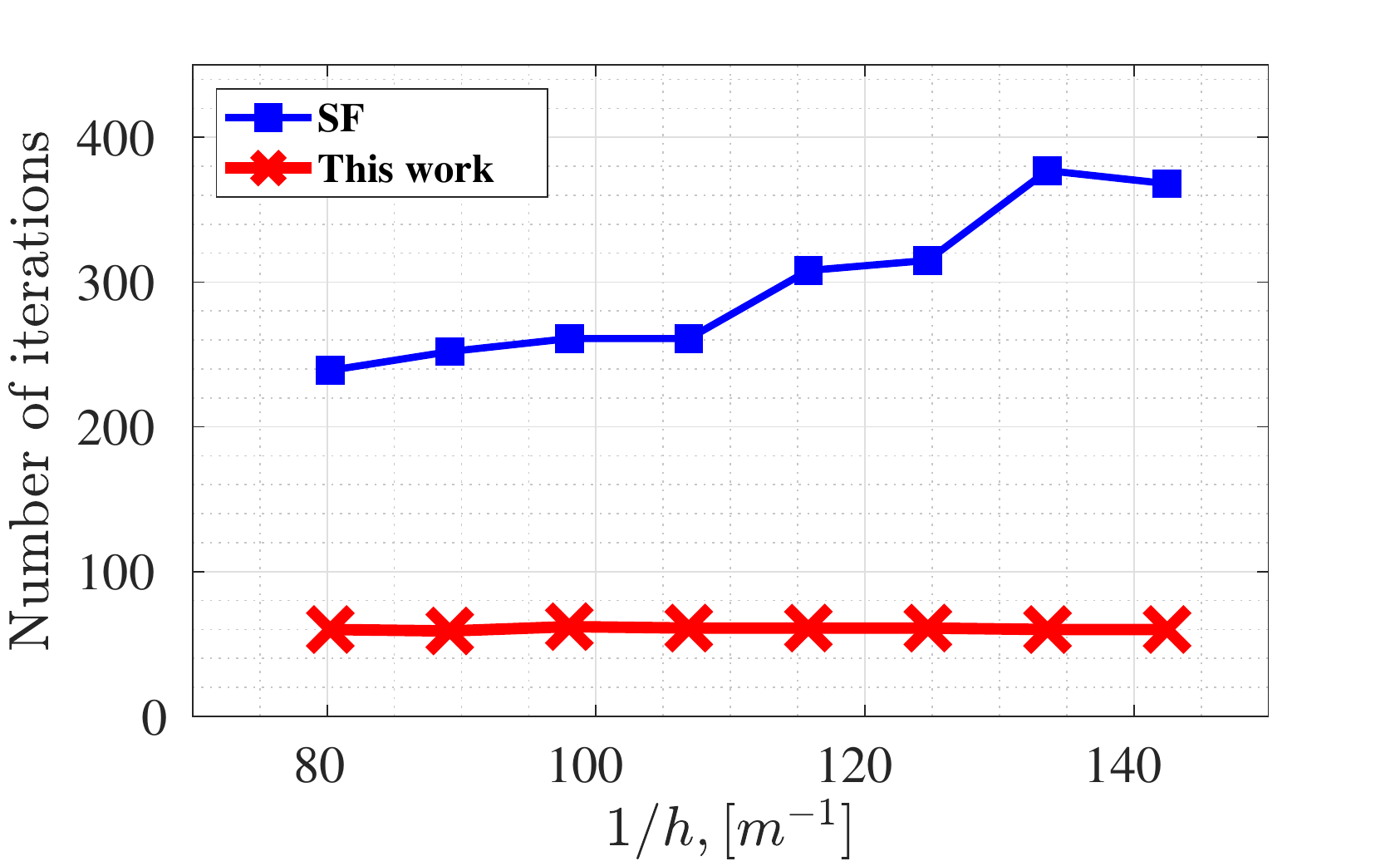}
}
\hfill
\subfloat[\label{subfig-2:sphere_NMVP}]{%
  \includegraphics[width=0.5\columnwidth]{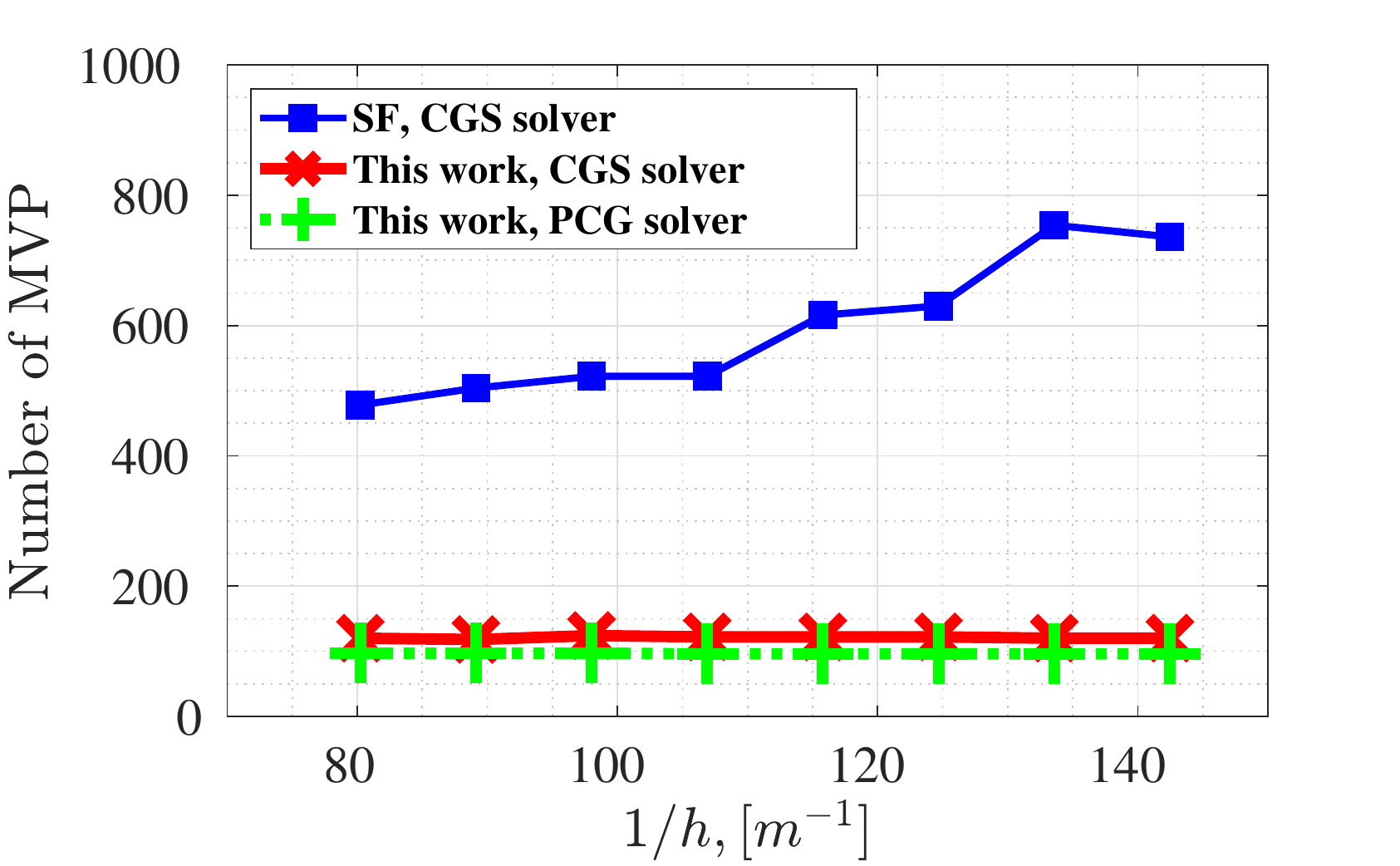}
}
\caption{(a) Number of iterations and (b) number of matrix-vector products (MVP) as a function of the inverse mesh refinement parameter $1/h$: comparison between the symmetric formulation (SF) and this work.}
\end{figure}

\begin{figure}
\subfloat[\label{subfig-1:sphere_NIt_CR}]{%
  \includegraphics[width=0.5
\columnwidth]{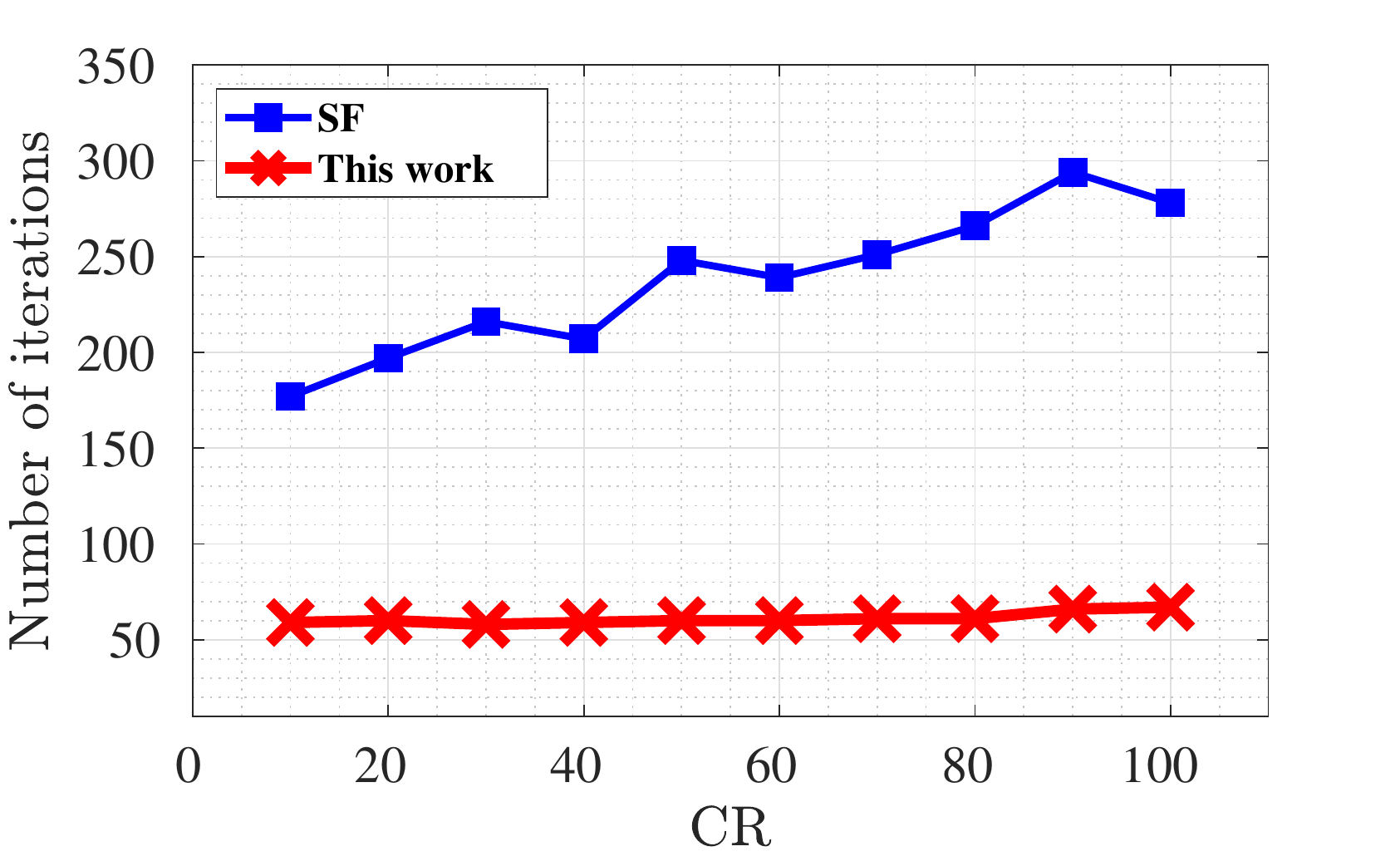}
}
\hfill
\subfloat[\label{subfig-2:sphere_NMVP_CR}]{%
  \includegraphics[width=0.5\columnwidth]{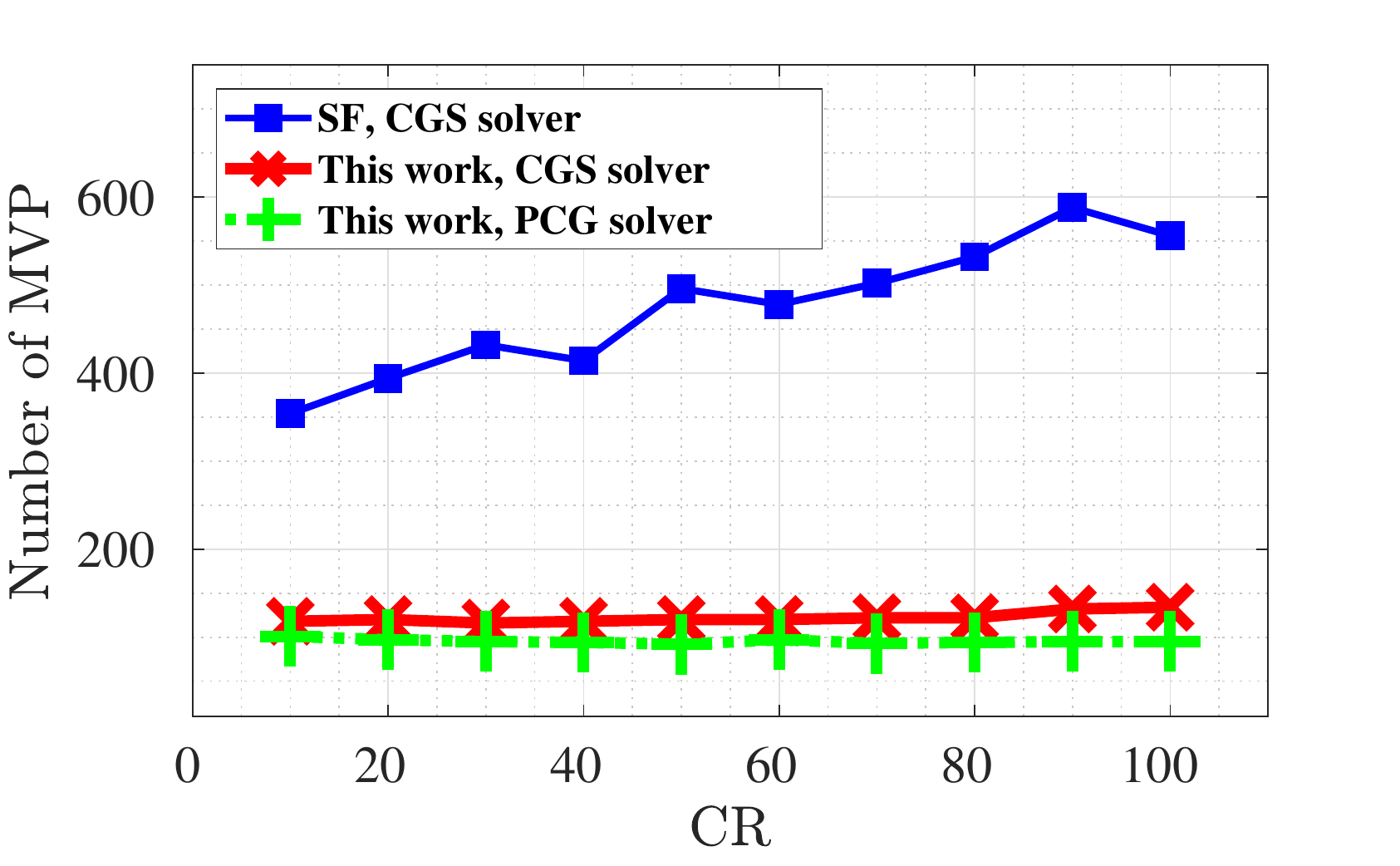}
}
\caption{(a) Number of iterations and (b) number of matrix-vector products (MVP) as a function of the conductivity ratio $CR$: comparison between the symmetric formulation (SF) and this work at $1/h\simeq$ \SI{80.2}{m^{-1}}.}
\end{figure}

\subsection{The MRI-obtained head model} \label{sec:results_MRI}
Subsequently, a realistic three-compartment head model obtained from MRI data has been considered. The boundaries of the geometry have been discretized by means of the meshes $\Gamma_1$, with $N_{C,1} = 3684$, $N_{V,1} = 1844$, $\Gamma_2$, with $N_{C,2} = 2334$, $N_{V,2} = 1169$, and  $\Gamma_3$, with $N_{C,3} = 2086$, $N_{V,3} = 1045$. The conductivities of the tissues have been set at $\sigma_1=$ \SI[parse-numbers = false]{\frac{1}{3}}{S/m},  $\sigma_2=$ \SI[parse-numbers = false]{\frac{1}{150}}{S/m}, and  $\sigma_3=$ \SI[parse-numbers = false]{\frac{1}{3}}{S/m}. The neural source has been modelled by a point dipole placed inside the inner compartment at a distance of approximately \SI{3.5}{cm} from $\Gamma_1$. \Cref{subfig-1:headA} shows the resulting potential distribution on the exterior layer. Moreover, the absolute difference of this result with respect to the one obtained from the unpreconditioned symmetric formulation is reported in correspondence of the position of $65$ electrodes placed on the scalp. The linear system of \num{10076} equations arising from the proposed formulation has been solved iteratively by means of the CGS solver in \num{155} iterations, to be compared with the \num{2157} iterations needed to solve the symmetric formulation by means of the same solver and by imposing an identical tolerance.

Given the excellent agreement in the results from the two formulations, the proposed scheme can clearly be employed in the evaluation of the lead-field matrix needed for the solution of the inverse EEG problem, at a reduced computational cost compared with the non-preconditioned one. The outcome of this test is shown in \Cref{subfig-2:headB}, where the neural source reconstructed from EEG measurements is represented. In particular, this has been obtained by applying the sLORETA inversion algorithm \citep{pascual-marqui2002standardized} to a lead-field matrix $\mat G \in \mathbb{R}^{N_E\times N_D}$, where $N_E = 65$ is the number of measurement points (corresponding to the number of electrodes) and $N_D = 19279$ is the number of test dipoles uniformly placed inside $\Omega_1$.

\begin{figure}
\subfloat[\label{subfig-1:headA}]{%
  \includegraphics[width=0.56
\columnwidth]{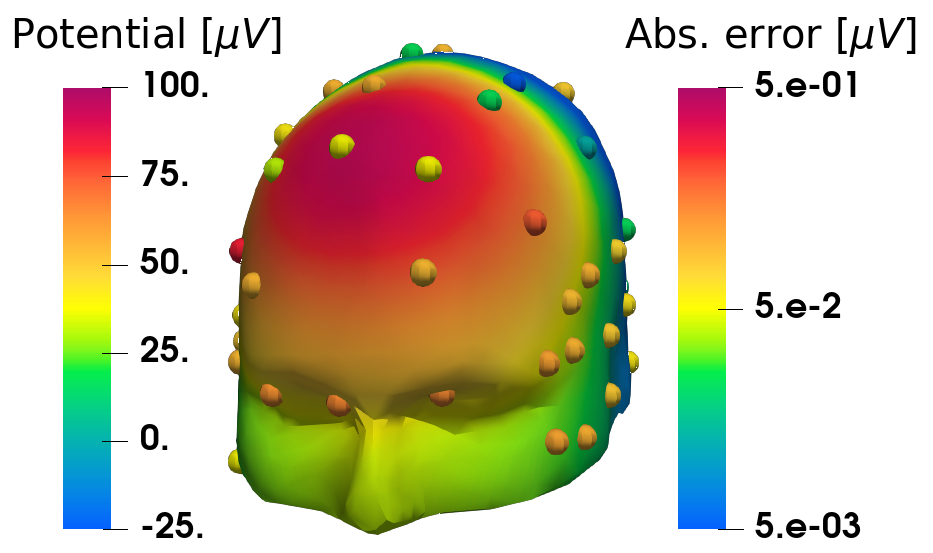}
}
\hfill
\subfloat[\label{subfig-2:headB}]{%
  \includegraphics[width=0.55\columnwidth]{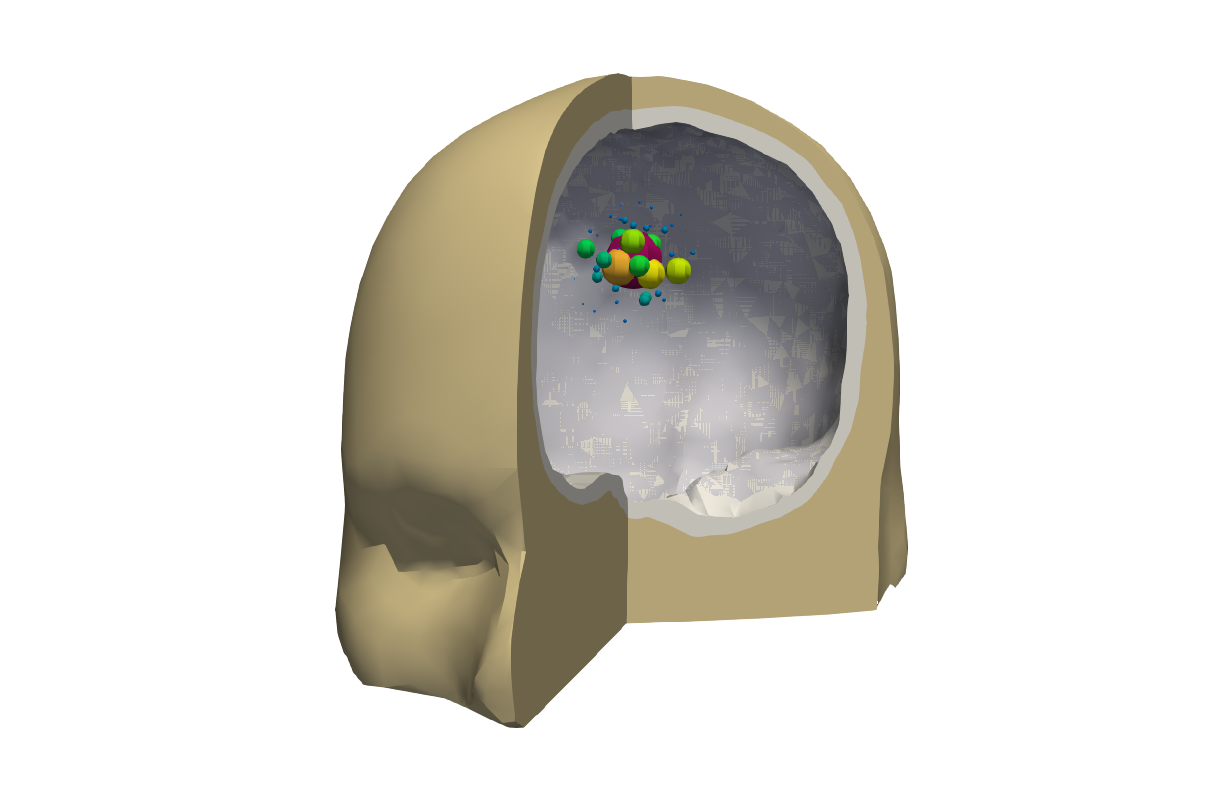}
}
\caption{(a) Scalp potential distribution and error at the electrodes with respect to the symmetric formulation solution. (b) Reconstructed epileptogenic source.}
\end{figure}

\appendix

\section{Proof of compactness of a block operator with compact blocks}
\label{sec:proof_blockcompact}
This section will provide a proof for the compactness of a $2\times 2$ block operator with compact blocks, that is, whose blocks are compact operators, rearranged from \citep{pillain2017line}. As a corollary, the compactness of a $N\times N$ block operator with compact blocks is shown.

Given the normed spaces $X_1$, $X_2$, ..., $X_N$, their Cartesian product, denoted as $X_1\bigoplus X_2 ...\bigoplus X_N$, equipped with the norm
\begin{equation}
||(x_1, x_2,...,n_N)||_{\bigoplus_{i=1}^N X_i} \coloneqq  \left( \sum_{i=1}^N ||x_i||^2_{X_i} \right)^{1/2},
\end{equation}
is their direct sum normed space \citep{abramovich2002invitation}, needed in the following derivations.

\begin{proposition}[]\label{Kd_comp}
Given the compact operators $\mathcal{K}_{11}:X\rightarrow X$ and $\mathcal{K}_{22}:Y\rightarrow Y$,
the block operator
\begin{equation}
\mathcal{K}_\mr{d} \coloneqq  \begin{pmatrix}
\mathcal{K}_{11} & 0 \\
0 & \mathcal{K}_{22}
\end{pmatrix}\,:\quad X \bigoplus Y \rightarrow X \bigoplus Y
\end{equation} is compact.
\end{proposition}
\begin{proof}
Let $\{u_n\}= \{(x_n,y_n)\}$ be a bounded sequence in $X\bigoplus Y$, that is, there exists a real positive constant $c$ such that
\begin{equation}
||u_n||^2_{X\bigoplus Y} = ||x_n||^2_X+ ||y_n||^2_Y\le c.
\end{equation}
This implies that $\{x_n\}$ and $\{y_n\}$ are bounded in $X$ and in $Y$. Therefore, by virtue of the compactness of $\mathcal{K}_{11}$, the sequence $\{\mathcal{K}_{11} x_n\}$ contains a convergent subsequence \citep[Theorem~1.2]{colton2013integral}, denoted by $\{\mathcal{K}_{11} x_{n_i}\}$.
Due to the boundness of $\{y_n\}$, also $\{y_{n_i}\}$ is bounded. Hence, $\{\mathcal{K}_{22} y_{n_i}\}$ contains a convergent subsequence, denoted as $\{\mathcal{K}_{22}  y_{n_{i_k}}\}$. Clearly, also $\{\mathcal{K}_{11} x_{n_{i_k}}\}$ is convergent, as subsequence of a convergent subsequence.

In symbols, $\forall \epsilon>0$, $\exists K$ such that $\forall k > K$
\begin{equation}
||\mathcal{K}_{11} x_{n_{i_k}} - l_x||_X < \epsilon \quad \text{and} \quad ||\mathcal{K}_{22} y_{n_{i_k}} - l_y||_Y < \epsilon.
\end{equation}
It follows that the application of the operator $\mathcal{K}_\mr{d}$ to the bounded sequence $\{u_n\}$, reading
\begin{equation}
\{\mathcal{K}_\mr{d} u_n\} = \{(\mathcal{K}_{11} x_n,\mathcal{K}_{22} y_n)\}
\end{equation}
contains the convergent subsequence $\{\mathcal{K}_\mr{d} u_{n_{i_k}}\} = \{ (\mathcal{K}_{11} x_{n_{i_k}},\mathcal{K}_{22} y_{n_{i_k}}) \}$. Indeed, $\forall k > K$
\begin{equation}
||\mathcal{K}_\mr{d} u_{n_{i_k}} - \begin{pmatrix} l_x\\l_y \end{pmatrix} ||_{X\bigoplus Y} = \left( || \mathcal{K}_{11} x_{n_{i_k}} - l_x ||^2_X + ||\mathcal{K}_{22} y_{n_{i_k}} - l_y||^2_Y\right) ^{1/2} \le \sqrt{2}\epsilon = \epsilon',
\end{equation}
which concludes the proof.
\end{proof}
\begin{proposition}[]\label{Kod_comp}
Given the compact operators $\mathcal{K}_{12}:Y\rightarrow X$ and $\mathcal{K}_{21}:X\rightarrow Y$,
the block operator
\begin{equation}
\mathcal{K}_\mr{od} \coloneqq  \begin{pmatrix}
0 & \mathcal{K}_{12} \\
\mathcal{K}_{21} & 0
\end{pmatrix}\,:\quad X \bigoplus Y \rightarrow X \bigoplus Y
\end{equation} is compact.
\end{proposition}
\begin{proof}
We follow similar steps as in the proof of \Cref{Kd_comp}. 
Let $\{u_n\}$ be a bounded sequence as above.
Given the compactness of $\mathcal{K}_{12}$, the sequence $\{\mathcal{K}_{12} y_n\}$ contains a converging subsequence, denoted by $\{\mathcal{K}_{12}  y_{n_i}\}$.
Then, we notice that the compact operator $\mathcal{K}_{21}$ applied to the bounded subsequence $x_{n_i}$ contains a converging subsequence, noted as $x_{n_{i_k}}$.
Finally, $\forall \epsilon>0$, $\exists K'$ such that $\forall k > K'$
\begin{equation}
||\mathcal{K}_{12} y_{n_{i_k}} - l_x'||_X < \epsilon \quad \text{and} \quad ||\mathcal{K}_{21} x_{n_{i_k}} - l_y'||_Y < \epsilon.
\end{equation}
Therefore, the sequence $\{\mathcal{K}_\mr{od} u_n\}$ contains a converging subsequence, denoted by $\{\mathcal{K}_\mr{od} u_{n_{i_k}}\}$. Indeed, $\forall k > K'$,
\begin{equation}
||\mathcal{K}_\mr{od} u_{n_{i_k}} - \begin{pmatrix} l_x'\\l_y' \end{pmatrix} ||_{X\bigoplus Y} = \left( || \mathcal{K}_{12} y_{n_{i_k}} - l_x' ||^2_X + ||\mathcal{K}_{21} x_{n_{i_k}} - l_y'||^2_Y\right) ^{1/2} \le \sqrt{2}\epsilon = \epsilon',
\end{equation}
hence $\mathcal{K}_\mr{od}$ is compact.
\end{proof}

\begin{theorem}[]
Given the compact operators $\mathcal{K}_{11}:X\rightarrow X$, $\mathcal{K}_{12}:Y\rightarrow X$, $\mathcal{K}_{21}:X\rightarrow Y$, and $\mathcal{K}_{22}:Y\rightarrow Y$,
the block operator
\begin{equation}
\mathcal{K} = \begin{pmatrix}
\mathcal{K}_{11} & \mathcal{K}_{12} \\
\mathcal{K}_{21} & \mathcal{K}_{22}
\end{pmatrix}\,:\quad X \bigoplus Y \rightarrow X \bigoplus Y
\end{equation}
is compact.
\end{theorem}

\begin{proof}
The operator $\mathcal{K}$ can be written as the sum of two operators involving the diagonal and the off-diagonal terms, named respectively $\mathcal{K}_\mr{d}$ and $\mathcal{K}_\mr{od}$,
\begin{equation}
\mathcal{K} = \underbrace{\begin{pmatrix}
\mathcal{K}_{11} & 0 \\
0 & \mathcal{K}_{22}
\end{pmatrix}}_\text{$\mathcal{K}_\mr{d}$} + \underbrace{\begin{pmatrix}
0 & \mathcal{K}_{12} \\
\mathcal{K}_{21} & 0
\end{pmatrix}}_\text{$\mathcal{K}_\mr{od}$}.
\end{equation}
The two operators $\mathcal{K}_\mr{d}$ and $\mathcal{K}_\mr{od}$ are compact, as shown in \Cref{Kd_comp} and \Cref{Kod_comp}. Therefore, $\mathcal{K}$ is compact as the sum of compact operators \citep[Theorem~1.4]{colton2013integral}.
\end{proof}

\begin{corollary}[]
Any block operator whose blocks are compact operators is compact.
\end{corollary}
\begin{proof}
A  $N\times N$ block operator whose blocks are compact can be decomposed as the summation of $(2N-1)$ block operators, each of them null apart one diagonal (principal or not), equal to the same diagonal of the original operator. For example, in the case $N = 3$,
\begin{align}
\mathcal{K} &= \begin{pmatrix}
\mathcal{K}_{11} & \mathcal{K}_{12} & \mathcal{K}_{13} \\
\mathcal{K}_{21} & \mathcal{K}_{22} & \mathcal{K}_{23} \\
\mathcal{K}_{31} & \mathcal{K}_{32} & \mathcal{K}_{33}
\end{pmatrix}  \nonumber\\
&= \begin{pmatrix}
0 & 0 & 0 \\
0 & 0 & 0 \\
\mathcal{K}_{31} & 0 & 0
\end{pmatrix}+
\begin{pmatrix}
0 & 0 & 0 \\
\mathcal{K}_{21} & 0 & 0 \\
0 & \mathcal{K}_{32} & 0
\end{pmatrix}+
\begin{pmatrix}
\mathcal{K}_{11} & 0 & 0 \\
0 & \mathcal{K}_{22} & \\
0 & 0 & \mathcal{K}_{33}
\end{pmatrix}+
\begin{pmatrix}
0 & \mathcal{K}_{12} & 0 \\
0 & 0 & \mathcal{K}_{23} \\
0 & 0 & 0
\end{pmatrix}+
\begin{pmatrix}
0 & 0 & \mathcal{K}_{13} \\
0 & 0 & 0 \\
0 & 0 & 0
\end{pmatrix}.
\end{align}
Then, the compactness of $\mathcal{K}$ can be shown by induction,
Indeed, since each term in this summation is compact for the same reasons outlined to prove the compactness of $\mathcal{K}_\mr{d}$ and $\mathcal{K}_\mr{od}$ in the $2\times 2$ operator case, the $N\times N$ block operator $\mathcal{K}$ is compact.
\end{proof}

\section{Analytic expression of the primal and dual Laplacian matrices}
\label{sec:appendix_dualLap}

We provide here the analytic expression of the elements of the matrices $\mat \Delta_i$ and $\tilde{\mat \Delta}_i$ discretizing the Laplace-Beltrami operator by means of pyramid and dual pyramid functions as an implementation aid.
In the following, we omit the subscript $_i$ that indicates the reference surface mesh $\Gamma_{h,i}$, its barycentric refinement $\bar{\Gamma}_{h,i}$, or its dual counterpart $\tilde{\Gamma}_{h,i}$ and that is applied to matrices, basis functions, and geometrical entities of the mesh, to simplify the notation.

By analytic evaluation of $\left( \nabla_\Gamma \lambda_m, \nabla_\Gamma \lambda_n \right)_{L^2(\Gamma_h)}$, the expression
\begin{equation}
\left(\LapLm\right)_{mn} = \begin{cases}
      \sum\limits_{c\in \mathit{Adj}(v_m)} \frac{\abs{e_{c,v_m}}^2}{4 \abs{c}} & \text{if $m=n$}\\
      \sum\limits_{c\in \mathit{Adj}(e_{mn})} \frac{\abs{e_{c,v_m}}\abs{e_{c,v_n}}}{4 \abs{c}} \cos(\theta_{c,mn}) & \text{if $v_m$, $v_n$ are connected by $e_{mn}$}\\
      0 & \text{otherwise}
    \end{cases}
    \label{eqn:Delta_anal}
\end{equation}
is retrieved, where $\mathit{Adj}(v_m)$ is the set of cells adjacent to the vertex $v_m$, $\mathit{Adj}(e_{mn})$ is the set of cells adjacent to the edge $e_{mn}$ connecting vertex $v_m$ and vertex $v_n$, $\abs{e_{c,v_m}}$ is the length of the edge of cell $c$ opposed to vertex $v_m$, $\abs{c}$ denotes the area of cell $c$.
As a general remark, the geometrical entities introduced up to now are elements of the primal mesh $\Gamma_h$, in symbols $v_m\in \Gamma_h$, $e_{c,v_{m}}\in \Gamma_h$, $e_{mn}\in \Gamma_h$, and $c\in \Gamma_h$.
The angle $\theta_{c,mn}$ is defined as
\begin{equation}
\theta_{c,mn} \coloneqq  \theta_{c,m}+\theta_{c,n}\,,
\end{equation}
where $\theta_{c,m}$ is the interior angle of cell $c$ at vertex $v_m$. In particular, given the length of the edges of $c$, it can be evaluated as 
\begin{equation}
\theta_{c,m} = \text{angle}\left(\abs{e_{mn}}, \abs{e_{c,v_n}}, \abs{e_{c,v_m}}\right)\, ,
\end{equation}
where
\begin{equation}
\text{angle}\left(\abs{e_1}, \abs{e_2}, \abs{e_3}\right) \coloneqq \arccos\left(\frac{\abs{e_1}^2+\abs{e_2}^2-\abs{e_3}^2}{2\abs{e_1}\abs{e_2}}\right)\,.
\end{equation}
returns the interior angle of a triangle opposed to its edge $e_3$ with $\abs{e_1}$, $\abs{e_2}$, and $\abs{e_3}$ denoting the length of the three sides of the triangle.
The notation employed in equation \eqref{eqn:Delta_anal} is shown in \Cref{fig:primalLap_notation}.

To define an analytic formula for the dual Laplacian matrix $\tilde{\mat \Delta}$, we express its elements as linear combinations of 
\begin{equation}
(\bar{\mat \Delta})_{mj,nk}=\left(\nabla_\Gamma \bar{\lambda}_{m,j},\nabla_\Gamma\bar{\lambda}_{n,k} \right)_{L^2(\bar{\Gamma_h})}
\end{equation}
resulting in
\begin{equation}
(\tilde{\mat \Delta})_{mn} = \sum\limits_{j=1}^7\sum\limits_{k=1}^7 \frac{1}{\mathit{NoC}(\bar{v}_{m,j})}\frac{1}{\mathit{NoC}(\bar{v}_{n,k})}(\bar{\mat \Delta})_{mj,nk},
\end{equation}
where the notation applied is the same as in \Cref{sec:further_notation} (without the mesh index subscript $_i$).
The analytic expression of $(\bar{\mat \Delta})_{mj,nk}$ is known from equation \eqref{eqn:Delta_anal}, as
\begin{equation}
(\bar{\mat \Delta})_{mj,nk} = \begin{cases}
      \sum\limits_{\bar{c}\in \mathit{Adj}(\bar{v}_{m,j})} \frac{\abs{\bar{e}_{\bar{c},\bar{v}_{m,j}}}^2}{4 \abs{\bar{c}}} & \text{if $m=n$ and $j=k$}\, ,\\
      \sum\limits_{\bar{c}\in \mathit{Adj}(\bar{e}_{mj,nk})} \frac{\abs{      \bar{e}_{\bar{c},\bar{v}_{m,j}}}\, \abs{\bar{e}_{\bar{c},\bar{v}_{n,k}}}}
      {4 \abs{\bar{c}}} \cos(\theta_{\bar{c},mj,nk}) & \text{if $\bar{v}_{m,j}$, $\bar{v}_{n,k}$ are connected by $\bar{e}_{mj,nk}$}\, ,\\
      0 & \text{otherwise.}
    \end{cases}
    \label{eqn:dualLapMatrix}
\end{equation}
As above, the geometrical entities belonging to the barycentrically refined mesh $\bar{\Gamma}_h$ are denoted with an upper bar. In particular $\mathit{Adj}(\bar{v}_{m,j})$ is the set of cells of $\bar{\Gamma}_h$ adjacent to the vertex $\bar{v}_{m,j}$, $\mathit{Adj}(\bar{e}_{mj,nk})$ is the set of cells of $\bar{\Gamma}_h$ adjacent to the edge $\bar{e}_{mj,nk}\in\bar{\Gamma}_h$ connecting the vertices $\bar{v}_{m,j}$, and $\bar{v}_{n,k}$, $\bar{e}_{\bar{c},\bar{v}_{m,j}}\in\bar{\Gamma}_h$ is the edge of cell $\bar{c}\in\bar{\Gamma}_h$ opposed to $\bar{v}_{m,j}\in\bar{\Gamma}_h$, $\abs{\bar{c}}$ is the area of cell $\bar{c}\in\bar{\Gamma}_h$. The angle $\theta_{\bar{c},mj,nk}$ is
\begin{equation}
\theta_{\bar{c},mj,nk} \coloneqq  \theta_{\bar{c},mj} + \theta_{\bar{c},nk},
\end{equation}
where $\theta_{\bar{c},mj}$ is the interior angle of the cell $\bar{c}$ at the vertex $\bar{v}_{m,j}$.

All the quantities introduced in equation \eqref{eqn:dualLapMatrix} are known from the geometrical properties of the primal mesh $\Gamma_h$.
For example, by denoting as $c \subset \Gamma_h$ the cell containing $\bar{c}\subset \bar{\Gamma}_h$, we recognize that $\abs{\bar{c}} = \abs{c}/6$, directly following from the properties of the barycentric refinement.
Moreover, the expression of the edge length $\abs{\bar{e}_{\bar{c},\bar{v}_{m,j}}}$ reads
\begin{equation}
\abs{\bar{e}_{\bar{c},\bar{v}_{m,j}}} = \begin{cases}
 \frac{1}{2}\abs{e_{\bar{c}}} & \text{if $\mathit{NoC}(\bar{v}_{m,j})=1$}\, ,\\
    \frac{2}{3}\abs{m_{\bar{c},{\text{vert}}}} & \text{if $\mathit{NoC}(\bar{v}_{m,j})=2$}\, ,\\
    \frac{1}{3}\abs{m_{\bar{c},{\text{side}}}} & \text{otherwise,}
\end{cases}
    \label{eqn:lapdual_edgedual}
\end{equation}
where $e_{\bar{c}}$ is the side of $c\supset\bar{c}$ with an infinite set of points in common with $\bar{c}$.
We denote by $m_{\bar{c},\text{vert}}$ the median of $c\supset\bar{c}$ such that the intersection $m_{\bar{c},\text{vert}} \cap \bar{c}$ contains infinite points, including a vertex of $c$.
The variable $m_{\bar{c},\text{side}}$ denotes the median of $c\supset\bar{c}$ such that the intersection $m_{\bar{c},\text{side}} \cap \bar{c}$ contains infinite points, but not including a vertex of $c$.
Finally, the angle $\theta_{\bar{c},mj}$ can be retrieved as
\begin{equation}
\theta_{\bar{c},mj} = 
\begin{cases}
    \text{angle}\left(\frac{2}{3}\abs{m_{\bar{c},\text{vert}}}, \frac{1}{3}\abs{m_{\bar{c},\text{side}}}, \frac{1}{2}\abs{e_{\bar{c}}}\right) & \text{if $\mathit{NoC}(\bar{v}_{m,j})=1$}\, ,\\
    \text{angle}\left(\frac{1}{3}\abs{m_{\bar{c},\text{side}}}, \frac{1}{2}\abs{e_{\bar{c}}}, \frac{2}{3}\abs{m_{\bar{c},\text{vert}}}\right) & \text{if $\mathit{NoC}(\bar{v}_{m,j})=2$}\, ,\\
    \text{angle}\left(\frac{1}{2}\abs{e_{\bar{c}}}, \frac{2}{3}\abs{m_{\bar{c},\text{vert}}}, \frac{1}{3}\abs{m_{\bar{c},\text{side}}}\right) & \text{otherwise.}
    \end{cases}
    \label{eqn:lapdual_theta_dual}
\end{equation}
Figures \ref{fig:pyr_not2}, \ref{fig:pyr_not3}, and \ref{fig:pyr_not4} represent the notation employed.

\begin{figure}
   \begin{minipage}[b]{0.48\textwidth}
     \centering
     \includegraphics[width=1\linewidth]{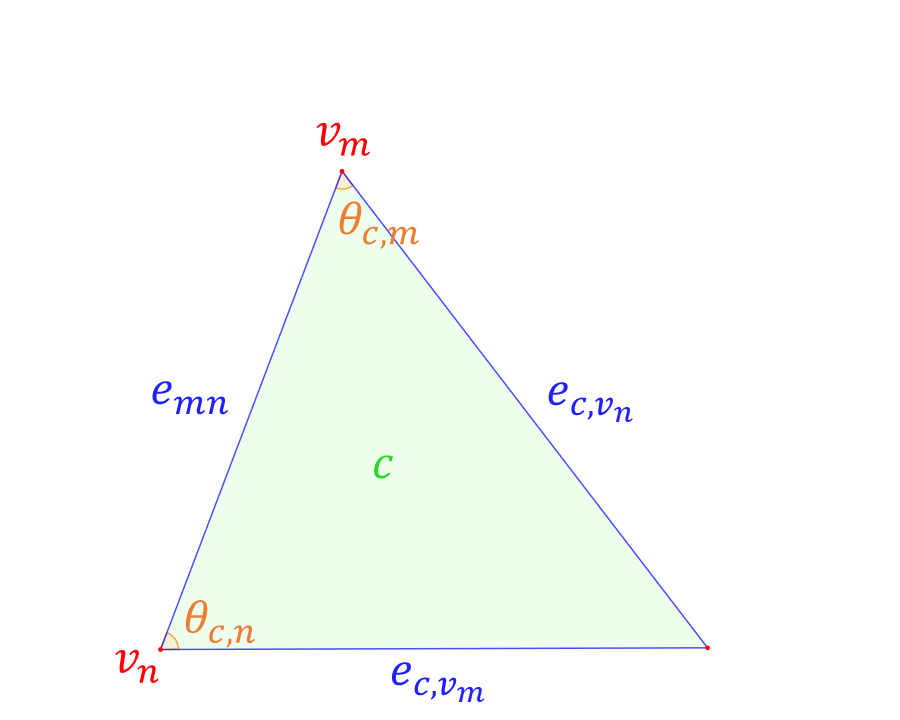}
     \caption{A primal cell in $\Gamma_h$ with the notation for the definition of $\left(\LapLm\right)_{mn}$ (equation \eqref{eqn:Delta_anal}).}
     \label{fig:primalLap_notation}
   \end{minipage}\hfill
   \begin{minipage}[b]{0.48\textwidth}
     \centering
     \includegraphics[width=1\linewidth]{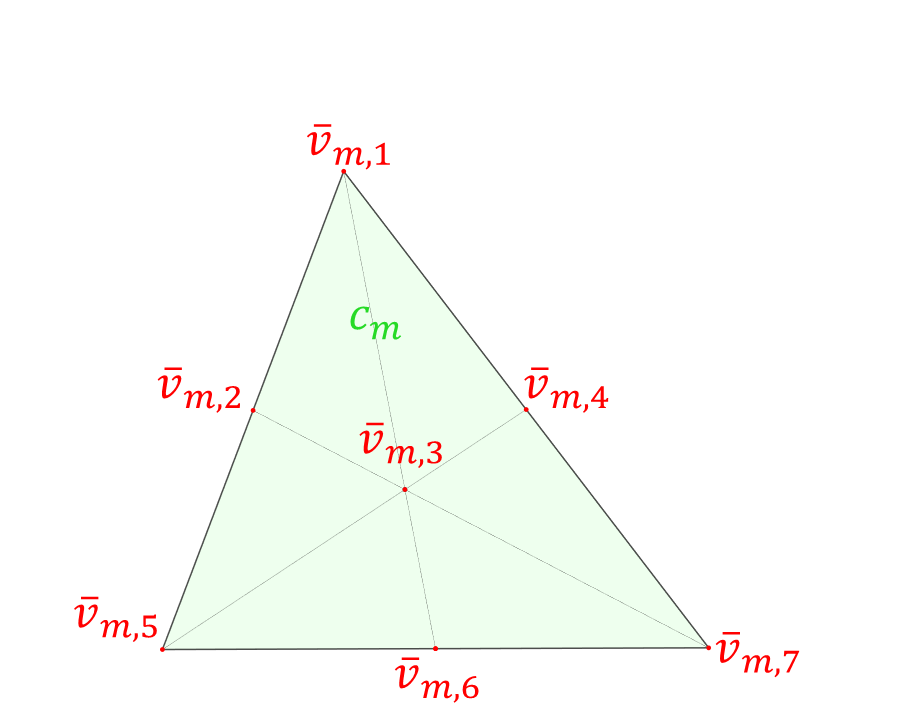}
     \caption{Seven vertices of $\bar{\Gamma}_i$ lie in the cell $c_m\in\Gamma_i$ (the numbering is randomly assigned).}
     \label{fig:pyr_not2}
   \end{minipage}
\end{figure}

\begin{figure}
   \begin{minipage}[b]{0.48\textwidth}
     \centering
     \includegraphics[width=1\linewidth]{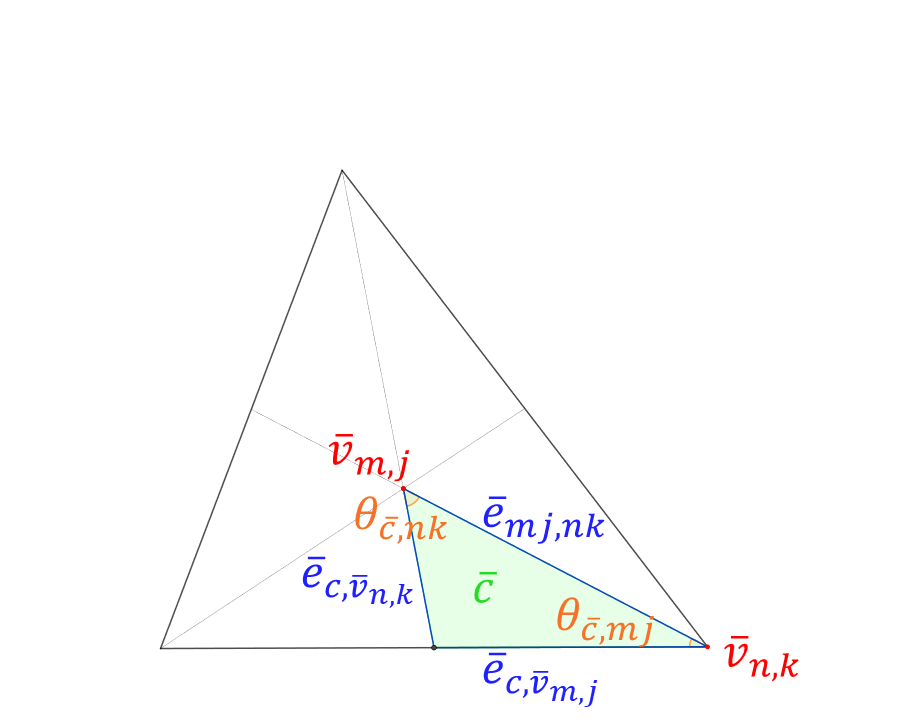}
     \caption{A cell in the barycentrically refined mesh $\bar{\Gamma}_h$ with the notation for the definition of $(\bar{\mat \Delta})_{mj,nk}$ (equation \eqref{eqn:dualLapMatrix}).}
     \label{fig:pyr_not3}
   \end{minipage}\hfill
   \begin{minipage}[b]{0.48\textwidth}
     \centering
     \includegraphics[width=1\linewidth]{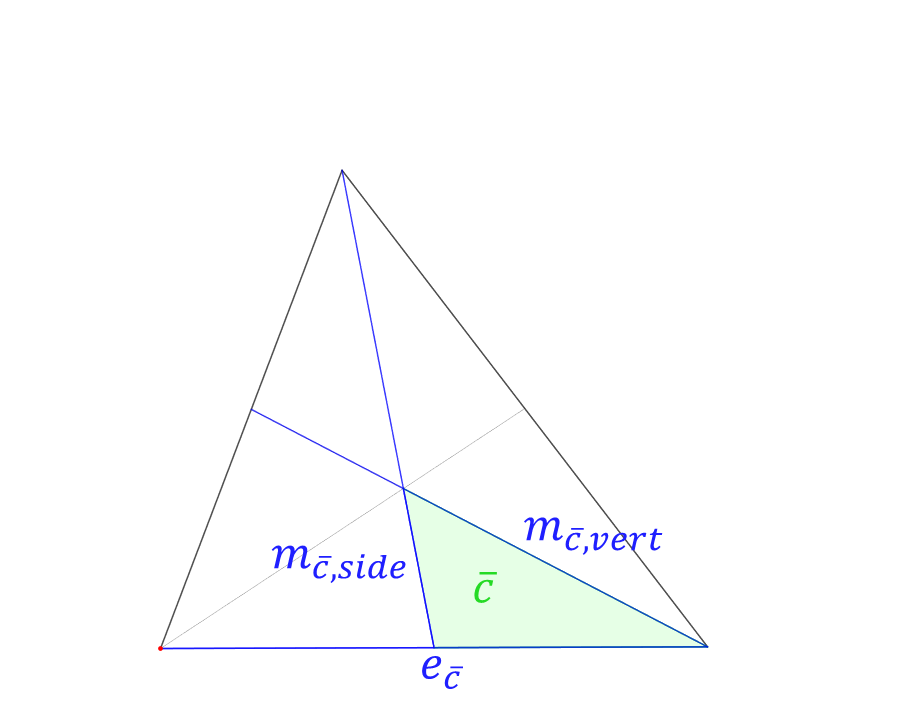}
     \caption{Notation for the definition of $\abs{\bar{e}_{\bar{c},\bar{v}_{m,j}}}$ (equation \eqref{eqn:lapdual_edgedual}) and $\theta_{\bar{c},mj}$ (equation \eqref{eqn:lapdual_theta_dual}).}
     \label{fig:pyr_not4}
   \end{minipage}
\end{figure}

\section*{Acknowledgment}
This work was supported in part by the European Research Council (ERC) through the European Union’s Horizon 2020 Research and Innovation Programme under Grant 724846 (Project 321), in part by the European Innovation Council (EIC) through the European Union’s Horizon Europe research Programme under Grant 101046748 (Project CEREBRO),  in part by the ANR Labex CominLabs under the project ``CYCLE'', and in part by the Italian Ministry of University and Research within the Program FARE, CELER, under Grant R187PMFXA4. 

\bibliographystyle{plainnat}

\bibliography{private_vgiunzioni.bib}

\end{document}

%% file: praemble/hft-paper-praemble-main.tex
\input{hft-paper-praemble-packages.tex}
\input{hft-macros.tex}

\input{sba-macros.tex}

%% file: praemble/hft-paper-praemble-packages.tex

\usepackage[utf8]{luainputenc}

\usepackage[T1]{fontenc}

\usepackage{graphicx}

\usepackage[english]{babel}
\addto\captionsenglish{}
\addto\captionsenglish{}
\usepackage{csquotes}

\usepackage{newtxtext}
\usepackage{amsthm}
\usepackage[slantedGreek]{newtxmath}
\usepackage[OMLmathsfit]{isomath}
\DeclareMathAlphabet{\mathbfsf}{\encodingdefault}{\sfdefault}{bx}{n}
\usepackage{bm}
\usepackage{mathtools}
\usepackage{commath}
\usepackage{siunitx}
\usepackage{IEEEtrantools}

\usepackage[caption=false,font=footnotesize]{subfig}

\usepackage{booktabs}
\usepackage{footmisc}  

\usepackage{url}

\input{hft-packages-theorem.tex}

\input{hft-packages-review.tex}
\input{hft-packages-tikz-pgfplots.tex}
\usepackage{hyperref} 
\input{hft-packages-cleveref.tex}

%% file: praemble/hft-packages-theorem.tex
\theoremstyle{definition}

\theoremstyle{plain}

\theoremstyle{remark}

%% file: praemble/hft-packages-review.tex
\usepackage{lineno}
\modulolinenumbers[5]
\usepackage{umoline}


%% file: praemble/hft-packages-tikz-pgfplots.tex
\usepackage{pgfplots}
\usepackage{pgfplotstable}
\pgfplotsset{compat=newest}
\pgfplotsset{plot coordinates/math parser=false}
\newlength\figureheight
\newlength\figurewidth
\pgfplotsset{every axis plot/.append style={line width=1.5pt},
    legend style={font=\footnotesize, 
        text height=1.0ex,
        draw=black,
        fill=white,
        legend cell align=left}}

%% file: praemble/hft-packages-cleveref.tex
\usepackage[english]{cleveref}

\Crefname{defn}{definition}{definitions}
\Crefname{defn}{Definition}{Definitions}

\Crefname{asm}{assumption}{assumptions}
\Crefname{asm}{Assumption}{Assumptions}

\crefname{lem}{lemma}{lemmas} 
\Crefname{lem}{Lemma}{Lemmas}

\crefname{prop}{proposition}{propositions} 
\Crefname{prop}{Proposition}{Propositions}

\crefname{thm}{theorem}{theorms} 
\Crefname{thm}{Theorem}{Theorms}

\crefname{cor}{corollary}{corollaries}
\Crefname{cor}{Corollary}{Corollaries}

%% file: praemble/hft-macros.tex
\input{hft-macros-environments.tex} 
\input{hft-macros-maths.tex} 
\input{hft-packages-acro.tex} 
\input{hft-macros-booktabs.tex}


%% file: praemble/hft-macros-environments.tex
\newcounter{subequation}
\newlength\mtabskip\mtabskip=-1.25cm

\def\mtabLong{long}
\makeatletter

\makeatother

%% file: praemble/hft-macros-maths.tex
\newcommand{\mr}{\mathrm}

\newcommand{\veg}[1]{\bm{#1}}     
\newcommand{\mat}[1]{\mathsfbfit{#1}} 
\renewcommand{\vec}[1]{\mathsfbfit{#1}} 
\newcommand{\op}[1]{\mathcal{#1}} 
\newcommand{\vecop}[1]{\bm{\mathcal{#1}}} 




\newcommand{\dd}{\mathrm{d}}  





\newcommand{\R}{\mathbb{R}}



\DeclareMathOperator{\cond}{cond}

\DeclareMathOperator{\diag}{diag}

\newcommand{\T}{\mr{T}}


\newcommand\restr[2]{{
        \left.\kern-\nulldelimiterspace 
        #1 
        \vphantom{|} 
        \right|_{#2} 
}}

\newcommand\rst[3]{{
        \left.\kern-\nulldelimiterspace 
        #1 
        \vphantom{|} 
        \right|_{#2}^{#3} 
}}



%

%% file: praemble/hft-packages-acro.tex
\usepackage{acro}

\DeclareAcronym{DG}
{
    short = DG ,
    long = discontinuous Galerkin
}

\DeclareAcronym{ACA}
{
    short = ACA ,
    long = adaptive cross approximation
}

\DeclareAcronym{EFIE}
{
    short =  EFIE ,
    long = electric field integral equation
}

\DeclareAcronym{MFIE}
{
    short =  MFIE ,
    long = magnetic field integral equation
}

\DeclareAcronym{CFIE}
{
    short =  CFIE ,
    long = combined field integral equation
}

\DeclareAcronym{MUIE}
{
    short =  MUIE ,
    long = Müller integral equation
}

\DeclareAcronym{PMCHWT}
{
    short =  PMCHWT ,
    long = Poggio-Miller-Chang-Harrington-Wu-Tsai integral equation
}

\DeclareAcronym{SPD}
{
    short =  SPD ,
    long = {symmetric, positive definite}
}

\DeclareAcronym{SPSD}
{
    short =  SPD ,
    long = {symmetric, positive semi-definite}
}

\DeclareAcronym{PEC}
{
    short =  PEC ,
    long = perfectly electrically conducting
}

\DeclareAcronym{RWG}
{
    short = RWG ,
    long = Rao-Wilton-Glisson
} 

\DeclareAcronym{BC}
{
    short = BC ,
    long = Buffa-Christiansen
}

\DeclareAcronym{SVD}
{
    short = SVD ,
    long = singular value decomposition
}

\DeclareAcronym{CG}
{
    short = CG ,
    long = conjugate gradient
} 

\DeclareAcronym{PCG}
{
    short = PCG ,
    long = preconditioned conjugate gradient
} 

\DeclareAcronym{CGS}
{
    short = CGS ,
    long = conjugate gradient squared
}

\DeclareAcronym{CMP}
{
    short = CMP ,
    long = Calderón multiplicative preconditioner
} 

\DeclareAcronym{RFCMP}
{
    short = RF-CMP ,
    long = refinement-free Calderón multiplicative preconditioner
} 

\DeclareAcronym{HPD}
{
    short = HPD ,
    long = {Hermitian, positive definite}
} 

\DeclareAcronym{RHS}
{
    short = RHS ,
    long = {right-hand side}
}

\DeclareAcronym{PW}
{
    short = PW ,
    long = {plane wave}
} 

\DeclareAcronym{HD}
{
    short = HD ,
    long = {Hertzian dipole}
} 

\DeclareAcronym{FF}
{
    short = FF ,
    long = {far-field}
} 

\DeclareAcronym{NF}
{
    short = NF ,
    long = {near-field}
} 

%% file: praemble/hft-macros-booktabs.tex

\newcolumntype {n}{c}
\newcolumntype {N}{>{\small}c}
\newcolumntype {L}{>{\small}l}
\newcolumntype {F}{>{\footnotesize}c}
\newcolumntype {v}[1]{>{\raggedright \hspace {0pt}} p {#1}}
\newcolumntype {V}[1]{>{\small \raggedright \hspace {0pt}} p {#1}}
\newcolumntype{d}[1]{>{\DC@{.}{.}{#1}}c<{\DC@end}}

%
\newcolumntype{R}[1]{%
    >{\begin{turn}{90}\begin{minipage}{#1}\small\raggedright\hspace{0pt}}l%
            <{\end{minipage}\end{turn}}%
}

%% file: praemble/sba-macros.tex
\NewDocumentCommand{\TA}{o}{
    \IfNoValueTF {#1} {%
        \vecop T_{\kern-2pt\mr{A}}
    }
    {
        \vecop T_{\kern-2pt\mr{A},#1}
    }
}

\NewDocumentCommand{\TPhi}{o}{
    \IfNoValueTF {#1} {%
        \vecop T_{\kern-2pt\Phiup}
    }
    {
        \vecop T_{\kern-2pt\Phiup,#1}
    }
}

\NewDocumentCommand{\matTA}{o}{
    \IfNoValueTF {#1} {%
        \mat T_\mr{A}   
        }
    {
        \mat T_{\mr{A},#1}
    }
}

\NewDocumentCommand{\matTPhi}{o}{
    \IfNoValueTF {#1} {%
        \mat T_\Phiup   
        }
    {
        \mat T_{\Phiup,#1}
    }
}

\NewDocumentCommand{\MSL}{o}{
    \IfNoValueTF {#1} {%
        \veg \Psi_\mr{SL}
        }
    {
        \veg \Psi_{\mr{SL},#1}
    }
}

\NewDocumentCommand{\MDL}{o}{
    \IfNoValueTF {#1} {%
        \veg \Psi_\mr{DL}
        }
    {
        \veg \Psi_{\mr{DL},#1}
    }
}

\NewDocumentCommand{\PA}{o}{
    \IfNoValueTF {#1} {%
        \veg \Psi_\mr{A}
        }
    {
        \veg \Psi_{\mr{A},#1}
    }
}

\NewDocumentCommand{\PPhi}{o}{
    \IfNoValueTF {#1} {%
        \veg \Psi_{\Phiup}
        }
    {
        \veg \Psi_{\Phiup,#1}
    }
}